\documentclass[12pt,oneside]{amsart}
\usepackage{amscd,amssymb, bbm,color}
\usepackage{graphicx, amsmath, amssymb, amscd, amsthm, euscript, psfrag, amsfonts,bm}

\usepackage{amssymb,amsfonts}
\usepackage[all,arc]{xy}
\usepackage{enumerate}
\usepackage{mathrsfs}
\usepackage{amsmath}
\usepackage{graphicx}
\usepackage{caption}
\usepackage{subcaption}

\newtheorem{thm}{Theorem}[section]
\newtheorem{cor}[thm]{Corollary}
\newtheorem{prop}[thm]{Proposition}
\newtheorem{lem}[thm]{Lemma}

\theoremstyle{definition}
\newtheorem{defn}[thm]{Definition}

\newtheorem{notn}[thm]{Notation}

\theoremstyle{remark}
\newtheorem{rem}[thm]{Remark}

\makeatletter
\let\c@equation\c@thm
\makeatother
\numberwithin{equation}{section}

\title[Effective equidistribution of circles]{Effective equidistribution of circles  in the limit sets of Kleinian groups}

\author{Wenyu Pan}

\address{Mathematics Department, Yale University, New Haven, CT 06520}

\email{wenyu.pan@yale.edu}

\begin{document}

\begin{abstract}
Consider a general circle packing $\mathcal{P}$ in the complex plane $\mathbb{C}$ invariant under a Kleinian group $\Gamma$. When $\Gamma$ is convex cocompact or its critical exponent is greater than 1, we obtain an effective equidistribution for small circles in $\mathcal{P}$ intersecting any bounded connected regular set in $\mathbb{C}$; this provides an effective version of an earlier work of Oh-Shah \cite{Asymptotic}. In view of the recent result of McMullen-Mohammadi-Oh \cite{MMO}, our effective circle counting theorem applies to the circles contained in the limit set of a convex cocompact but non-cocompact Kleinian group whose limit set contains at least one circle. Moreover consider the circle packing $\mathcal{P}(\mathcal{T})$ of the ideal triangle attained by filling in largest inner circles. We give an effective estimate to the number of disks whose hyperbolic areas are greater than $t$, as $t\to0$, effectivising the work of Oh \cite{Harmonic analysis}.
 \end{abstract}

\maketitle

\section{Introduction}
A circle packing in the complex plane $\mathbb{C}$ is simply a countable union of circles (here a line is regarded as a circle of infinite radius).  Compared to the conventional definition of a circle packing, our definition is more general as circles are allowed to intersect each other. Given a circle packing $\mathcal{P}$, 
we seek to estimate the number of small circles intersecting a bounded subset in $\mathbb{C}$ (see Figure 1 for examples).  

\begin{figure}[h]
\centering
\begin{subfigure}{.5\textwidth}
  \centering
  \includegraphics[width=.6\linewidth]{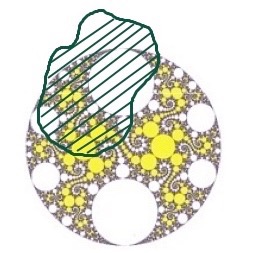}
  \label{fig:sub1}
\end{subfigure}%
\begin{subfigure}{.5\textwidth}
  \centering
  \includegraphics[width=.6\linewidth]{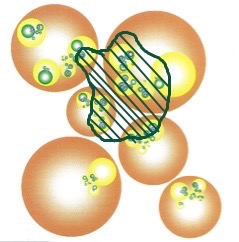}
  \label{fig:sub2}
\end{subfigure}
\caption{Circle packing intersecting bounded region (background pictures are reproduced from Indra's Pearls: The Vision of Felix Klein, by D. Mumford, C. Series and D. Wright, copyright Cambridge University Press 2002)}
\label{fig:test}
\end{figure}

Assume $\mathcal{P}$ is locally finite, i.e., for any $T>1$, there are only finitely many circles in $\mathcal{P}$ of Euclidean curvature at most $T$ intersecting any fixed bounded subset in $\mathbb{C}$. For a bounded subset $E$ in $\mathbb{C}$ and $T>1$, we set
\begin{equation*}
N_T(\mathcal{P},E):=\#\{C\in \mathcal{P}: C\cap E \neq \emptyset, \operatorname{Curv}(C)<T\},
\end{equation*}
where $\operatorname{Curv}(C)$ denotes the Euclidean curvature of $C$. As $\mathcal{P}$ is locally finite, $N_T(\mathcal{P},E)<\infty$.

In \cite{Asymptotic}, Oh and Shah considered a very general locally finite circle packing $\mathcal{P}$: suppose $\mathcal{P}$ is invariant under a torsion-free non-elementary geometrically finite Kleinian group $\Gamma<\operatorname{PSL}_2(\mathbb{C})$ \footnote{In the rest of the paper, we always assume the Kleinian groups we consider are torsion-free and non-elementary}. They obtained an asymptotic estimate to $N_T(\mathcal{P},E)$. In particular, they introduced a locally finite Borel measure $\omega_{\Gamma}$ on $\mathbb{C}$ determined by $\Gamma$ (Definition \ref{measure on the complex plane})  such that under some further assumption on $\Gamma$, we have
\begin{equation*}
\lim_{T\to \infty}\frac{N_T(\mathcal{P},E_1)}{N_T(\mathcal{P},E_2)}=\frac{\omega_{\Gamma}(E_1)}{\omega_{\Gamma}(E_2)},
\end{equation*}
where $E_1$ and $E_2$ are any bounded Borel sets in $\mathbb{C}$ satisfying $\omega_{\Gamma}(\partial(E_1))=\omega_{\Gamma}(\partial(E_2))=0$.

In this paper, we extend Oh-Shah's result and provide an effective estimate to $N_T(\mathcal{P},E)$. To apply our theorem, we need to impose a more stringent condition on $E$:  we require not only $\omega_{\Gamma}(\partial (E))=0$ but the $\epsilon$-neighborhood of $\partial(E)$ is of small size. Sets satisfying such property will be called regular (Definition \ref{regularity condition}). Denote the critical exponent of $\Gamma$ by $\delta_{\Gamma}$. We show the following.
\begin{thm}
\label{main thm}
Assume $\mathcal{P}$ is a locally finite circle packing invariant under a geometrically finite Kleinian group $\Gamma$ and with finitely many $\Gamma$-orbits. When $\delta_{\Gamma}\leq 1$, we assume further that $\Gamma$ is convex cocompact.  Then for any  bounded connected regular set $E\subset \mathbb{C}$, there exists  $\eta>0$, such that as $T \to \infty$,
\begin{equation*}
N_T(\mathcal{P},E)=c\cdot\omega_{\Gamma}(E)\cdot T^{\delta_{\Gamma}}+O(T^{\delta_{\Gamma}-\eta}),
\end{equation*}  
where $c>0$ is a constant depending only on $\Gamma$ and $\mathcal{P}$.
\end{thm}

    


Denote by $\Lambda (\Gamma)\subset \mathbb{C}\cup \{\infty\}$ the limit set of $\Gamma$ which is the set of accumulation points of an orbit of $\Gamma$ in $\mathbb{C}\cup \{\infty\}$ under the linear fractional transformation action. When $\Gamma$ is convex cocompact or it has no rank 2 cusps with $\delta_{\Gamma}>1$,  the measure $\omega_{\Gamma}(E)$  in Theorem \ref{main thm} equals the $\delta_{\Gamma}$-dimensional Hausdorff measure of $E\cap\Lambda (\Gamma)$ \cite{Sullivan}. 

\subsection*{Circles in the limit set of a Kleinian group}

   Suppose $\Gamma$ is convex cocompact. Consider the set of circles contained in $\Lambda(\Gamma)$:
   \begin{equation*}
   \mathcal{I}(\Gamma):=\{C\subset \Lambda(\Gamma)\}.
   \end{equation*}
   McMullen, Mohammadi and Oh  showed that if $\Lambda(\Gamma)\ne \mathbb C\cup\{\infty\}$,
there are only finitely many $\Gamma$-orbits of circles in $\mathcal{I}(\Gamma)$,
and each such circle arises from a compact $\operatorname{PSL_2(\mathbb R)}$-orbit (Corollary 11.3 and Theorem B.1 in \cite{MMO}); this implies that $\mathcal{I}(\Gamma)$ is a locally finite circle packing with finitely many $\Gamma$-orbits and hence
Theorem \ref{main thm} applies to $\mathcal I(\Gamma)$.

 \begin{cor}
 \label{circles in the limit set}
Let $\Gamma$ be a convex cocompact Kleinian subgroup with critical exponent $\delta_{\Gamma}<2$.  Assume that $\mathcal I(\Gamma)$ is non-empty.
Then for any bounded connected regular set $E\subset\mathbb{C}$ , there exists $\eta>0$, such that as $T\to \infty$,
\begin{equation*}
N_T(\mathcal{I}(\Gamma),E)=c\cdot \mathcal{H}^{\delta_{\Gamma}}(E\cap \Lambda (\Gamma))\cdot T^{\delta_{\Gamma}}+O(T^{\delta_{\Gamma}-\eta}),
\end{equation*}
where $c>0$ is a constant depending only on  $\Gamma$, and  $\mathcal{H}^{\delta_{\Gamma}}(E\cap \Lambda (\Gamma))$ is the $\delta_{\Gamma}$-dimensional Hausdorff measure of $E\cap \Lambda (\Gamma)$.
\end{cor}

\subsection*{Circles in  an ideal triangle of $\mathbb{H}^2$}
Let $\mathcal{T}$ be an ideal triangle in the hyperbolic plane $\mathbb{H}^2$, i.e., a triangle whose sides are hyperbolic lines connecting vertices on the geometry boundary $\partial{\mathbb{H}}^2$. Such an ideal triangle exists and is unique up to hyperbolic isometries. Consider the circle packing $\mathcal{P}(\mathcal{T})$ in $\mathcal{T}$ attained by filling in the largest inner circles (see Figure 2). We give an effective estimate to the number of disks enclosed by circles in $\mathcal{P(\mathcal{T})}$ whose hyperbolic areas are greater than $t$. 

Let $\overline{\mathcal{P}(\mathcal{T})}$ be the closure of $\mathcal{P}(\mathcal{T})$. The Hausdorff dimension of $\overline{\mathcal{P}(\mathcal{T})}$, denoted by $\alpha$, equals the residual dimension of an Apollonian circle packing \cite{McMullen}. For $C\in \mathcal{P}(\mathcal{T})$, let $\operatorname{Area}_{\operatorname{hyp}}(C)$  be the hyperbolic area of the disk enclosed by $C$. 
\begin{thm} 
\label{main thm 1}
There exist $c>0$ and  $\eta>0$, such that  as $t \to 0$,
\begin{equation*}
\#\{C\in \mathcal{P}(\mathcal{T}):\operatorname{Area}_{\operatorname{hyp}}(C)>t\}=c\, t^{-\frac{\alpha}{2}}+O(t^{-\frac{\alpha}{2}+\eta}).
\end{equation*}
\end{thm}

\begin{figure}[h]
\includegraphics[scale=0.4]{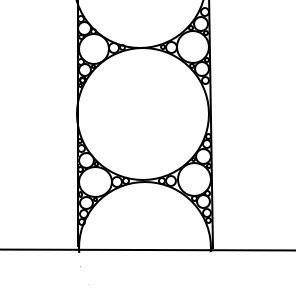}
\caption{Circle packing in ideal hyperbolic triangle}
\end{figure}

The asymptotic formula for this counting problem without a rate was obtained by Oh in \cite{Harmonic analysis}.

\subsection*{On the proof of Theorems \ref{main thm} and \ref{main thm 1} }
Our proof of Theorem \ref{main thm} is built on the approach employed in \cite{Asymptotic}, while providing an effective statement
for each step of their arguments and combining it with the effective equidistribution results in \cite{Matrix coefficients}.

Theorem \ref{main thm 1} does not immediately follow  from Theorem \ref{main thm} since the ideal triangle is not bounded in the hyperbolic space. In order to
prove Theorem \ref{main thm 1},  we need to obtain an effective estimate to the $\alpha$-dimensional Hausdorff measure in hyperbolic metric of neighborhoods of the vertices of the ideal triangle, as well as to give an effective estimate to the number of circles in such neighborhoods.

We refer readers to \cite{Lee&Oh} and \cite{Vinogradov} for effective counting results when $\mathcal{P}$ is an Apollonian circle packing. See also \cite{ParPau} for related counting results.

\subsection*{Acknowledgments} 

I would like to thank my advisor, Hee Oh, for suggesting the problem as a part of my thesis and for continued guidance and support. I would also like to thank Dale Winter for useful discussion, as well as comments on an earlier draft of the paper.

\section{Preliminaries}
\label{Preliminaries}

\subsection{Notations}

 We let
 \begin{align*}
 G&:=\operatorname{PSL}_{2}(\mathbb{C}),\,\,K:=\operatorname{PSU}(2),\\
 M&:=\{m_{\theta}:=\begin{pmatrix} e^{i\theta} & 0 \\ 0 & e^{-i\theta}\end{pmatrix}:\theta \in \mathbb{R}\},\\
 H&:=\operatorname{PSU}(1,1)\cup \begin{pmatrix} 0 & 1 \\-1 & 0\end{pmatrix} \operatorname{PSU}(1,1).
 \end{align*}
 
 Let $\mathbb{H}^3$ be the hyperbolic 3-space. We use  the following coordinates for the upper half space model of $\mathbb{H}^3$:
  \begin{equation*}
  \mathbb{H}^3=\{z+jy:z\in \mathbb{C}, y \in \mathbb{R}_{> 0}\},
  \end{equation*}
  where $j=(0,1)$. 
  
  The geometric boundary $\partial\mathbb{H}^3$ is the extended complex plane $\hat{\mathbb{C}}$. The group of orientation preserving isometries of $\mathbb{H}^3$ is given by $G$. 
 Noting that $G$ acts transitively on $\mathbb{H}^3$  and  $K=\operatorname{Stab}_{G}(j)$, we identify $\mathbb{H}^3$ with $G/K$ via the map $[g] \mapsto gj$.

 Let $T^{1}(\mathbb{H}^3)$ be the unit tangent bundle of $\mathbb{H}^3$, and $X_0\in T^1(\mathbb{H}^3)$ the upward unit normal vector based at $j$. Then  $T^1(\mathbb{H}^3)$ can be identified with $G/M$ via the map $gX_0 \mapsto [g]$ as $M=\operatorname{Stab}_{G}(X_0)$.

Any circle $\mathit{C}$ in $\hat{\mathbb{C}}$ determines a unique totally geodesic plane in $\mathbb{H}^3$, denoted by $\hat{C}$. Let $C^{\dagger}$ be the unit normal bundle of $\hat{C}$. The group $\operatorname{Stab}_{G}(\hat{C})$ acts transitively on both $\hat{C}$ and $C^{\dagger}$. In particular, if we denote by $C_0$ the unit circle centered at the origin,  then $\hat{C_0}$ and $C_0^{\dagger}$  can be identified with $H/H\cap K$ and $H/M$ respectively as $H=\operatorname{Stab}_{G}(\hat{C_0})$.

\subsection{Measures on $\Gamma \backslash G/M$ } Let $\Gamma<G$  be a  geometrically finite Kleinian group. A family of finite measures $\{\mu_x:x\in \mathbb{H}^3\}$ is called a $\Gamma$-invariant conformal density of dimension $\delta_{\mu}>0$, if each $\mu_x$ is a non-zero finite Borel measure on $\partial\mathbb{H}^3$ satisfying for any $x,y \in \mathbb{H}^{3}$, $\xi \in \partial \mathbb{H}^3$ and $\gamma \in \Gamma$,
\begin{equation*}
\gamma_{*}\mu_x=\mu_{\gamma x} \quad \operatorname{and} \quad  \frac{d\mu_y}{d\mu_x}(\xi)=e^{-\delta_{\mu}\beta_{\xi}(y,x)}
\end{equation*}
where $\gamma_{*}\mu_x(F):=\mu_x(\gamma^{-1}(F))$ for any Borel subset $F$ of $\partial \mathbb{H}^3$. Here $\beta_{\xi}(y,x)$ is the Busemann function given by $\beta_{\xi}(y,x)=\lim_{t\to \infty} d(\xi_t,y)-d(\xi_t,x)$, where $\xi_t$ is any geodesic ray tending to $\xi$.

We denote by $\{\nu_{\Gamma,x}:x\in \mathbb{H}^3\}$ (or simply $\{\nu_x:x\in\mathbb{H}^3\}$) the Patterson-Sullivan density (or PS-density), which is a $\Gamma$-invariant conformal density of dimension $\delta_{\Gamma}$ with $\delta_{\Gamma}$  the critical exponent of $\Gamma$. We will denote the critical exponent of $\Gamma$ simply by $\delta$ when there is no room for confusion. Denote by $\{m_x:x \in \mathbb{H}^3\}$  the Lebesgue density, which is a $G$-invariant conformal density on the boundary $\partial \mathbb{H}^3$ of dimension 2.

 Let $\pi: T^{1}(\mathbb{H}^3)\to \mathbb{H}^3$  be the canonical projection map. For $u\in T^{1}(\mathbb{H}^3)$, denote by $u^{\pm}\in \partial\mathbb{H}^3$ the forward and the backward endpoints of the geodesic determined by $u$. The Hopf parametrization $u \mapsto (u^{+},u^{-}, s:=\beta_{u^{-}}(j,\pi(u)))$ gives a homeomorphism between $T^1(\mathbb{H}^3)$ and $(\partial\mathbb{H}^3 \times \partial\mathbb{H}^3)\backslash\{(\xi,\xi):\xi \in \partial\mathbb{H}^3 \} \times \mathbb{R}$. Using the identification of $G/M$ with $T^{1}(\mathbb{H}^3)$, we define the Bowen-Margulis-Sullivan measure $\tilde{m}^{\operatorname{BMS}}_{\Gamma}$ and Burger-Roblin measure $\tilde{m}^{\operatorname{BR}}_{\Gamma}$ on $G/M$ as follows:

\begin{defn}
Set
\begin{enumerate}
\item $d\tilde{m}^{\operatorname{BMS}}_{\Gamma}(g)=e^{\delta \beta_{gX_0^{+}}(j,gj)}e^{\delta \beta_{gX_0^{-}}(j,gj)}d\nu_j(gX_0^{+})d\nu_j(gX_0^{-})ds$;
\item $d\tilde{m}^{\operatorname{BR}}_{\Gamma}(g)=e^{2\beta_{gX_0^{+}}(j,gj)}e^{\delta \beta_{gX_0^{-}}(j,gj)}dm_j(gX_0^{+})d\nu_j(gX_0^{-})ds$.
\end{enumerate}
\end{defn}
Both $\tilde{m}^{\operatorname{BMS}}_{\Gamma}$ and $\tilde{m}^{\operatorname{BR}}_{\Gamma}$ are left $\Gamma$-invariant by the properties of conformal density. They induce locally finite Borel measures on $\Gamma\backslash G/M$, which we will denote by $m^{\operatorname{BMS}}_{\Gamma}$ and $m^{\operatorname{BR}}_{\Gamma}$. When there is no room for confusion, we will write $\tilde{m} ^{\operatorname{BMS}}$, $\tilde{m}^{\operatorname{BR}}$, $m^{\operatorname{BMS}}$ and $m^{\operatorname{BR}}$ for simplicity.

\subsection{Measures on $(\Gamma \cap H) \backslash H/M$ and $\operatorname{sk}_{\Gamma}(\mathcal{P})$}
Denote $\Gamma\cap H$ by $\Gamma_{H}$ for convenience. Using the PS-density $\{\nu_x\}$, 
we construct the measure $\tilde{\mu}^{\operatorname{PS}}_{\Gamma, H}$ 
on $H/M=C^{\dagger}_0$:
\begin{equation*}
d\tilde{\mu}^{\operatorname{PS}}_{\Gamma,H}:=e^{\delta\beta_{(hX_0)^+}(j,hj)}d\nu_j(hX_0^+).
\end{equation*}
Note that  $d\tilde{\mu}^{\operatorname{PS}}_{\Gamma,H}$ 
is left $\Gamma_{H}$-invariant. Hence it induces  a locally finite Borel measure  on $\Gamma_{H} \backslash H/M=\Gamma_{H}\backslash C^{\dagger}_0$. We denote the induced measure by  $d\mu_{\Gamma, H}^{\operatorname{PS}}$ 
When there is no ambiguity about $\Gamma$, we simply write $d\mu^{\operatorname{PS}}_{H}$. 

We introduce a measure associated with a circle packing $\mathcal{P}$:
\begin{defn}[The $\Gamma$-skinning size of $\mathcal{P}$] 
\label{skinning size of circle packing}
For a circle packing $\mathcal{P}$ in $\hat{\mathbb{C}}$ consisting of finitely many $\Gamma$-orbits, define $0 \leq \operatorname{sk}_{\Gamma}(\mathcal{P}) \leq \infty$ as follows:
\begin{equation*}
\operatorname{sk}_{\Gamma}(\mathcal{P}):= \sum_{i \in I} |\mu_{g_{C_i}^{-1}\Gamma g_{C_i},H}^{\operatorname{PS}}|,
\end{equation*}
where $\{C_i: i\in I\}$ is a set of representatives of $\Gamma$-orbits in $\mathcal{P}$ and $g_{C_i}\in G$ is an element such that $g_{C_i}(C_0)=C_i$.
\end{defn}

\begin{rem}

 It is shown in \cite{Sullivan} that $\{\nu_{x}:x\in \mathbb{H}^3\}$ is unique up to scalars.  Using this property, we can verify that Definition \ref{skinning size of circle packing} does not depend on the choice of $g_{C_i}$.


\end{rem}

\begin{thm}[Theorem 2.4 and Lemma 3.2 in \cite{Asymptotic}]
Assume that $\Gamma$ is either convex cocompact or its critical exponent $\delta$ is greater than 1. Let $\mathcal{P}$ be a locally finite circle packing in $\hat{\mathbb{C}}$ invariant under $\Gamma$ with finitely many $\Gamma$-orbits. Then $\operatorname{sk}_{\Gamma}(\mathcal{P})<\infty$.
\end{thm}

\begin{prop}
For any circle $C$, if $\Gamma (C)$ is a locally finite circle packing consisting of infinitely many circles, then $\operatorname{sk}_{\Gamma}(\Gamma (C))>0$.
\end{prop}

\begin{proof}
As $\Gamma(C)$ consists of infinitely many circles, we have $[\Gamma: \Gamma\cap \operatorname{Stab}_{G}(C)]=\infty$ (Lemma 3.1 in \cite{Asymptotic}). The local finiteness of $\Gamma(C)$ implies that the map $\Gamma\cap \operatorname{Stab}_{G}(C)\backslash \hat{C}\to \Gamma\backslash \mathbb{H}^3$ is proper (Lemma 3.2 in \cite{Asymptotic}). The proposition follows from Proposition 6.7 in \cite{Equidistribution}.
\end{proof}

\section{Effective counting for general circle packings}
Throughout this section, we set $\Gamma<\operatorname{PSL}_2(\mathbb{C})$  to be a torsion-free non-elementary geometrically finite Kleinian group. Let $\mathcal{P}$ be a locally finite circle packing consisting of finitely many $\Gamma$-orbits.

Let
\begin{equation*} 
N=\{n_z:=\begin{pmatrix} 1 & z\\ 0 & 1\end{pmatrix}:z\in \mathbb{C}\}\;\;\text{and}\;\;N^{-}=\{n^{-}_{w}:=\begin{pmatrix} 1 & 0\\w & 1\end{pmatrix}:w\in \mathbb{C}\}.
\end{equation*}
For a subset $E\subset \mathbb{C}$, set
\begin{equation*}
N_{E}:=\left \{ n_z : z\in E \right \}.
\end{equation*}
Let 
\begin{equation*}
A=\{a_t:=\begin{pmatrix} e^{\frac{t}{2}} & 0\\ 0 & e^{-\frac{t}{2}}\end{pmatrix}:t\in \mathbb{R}\}\;\;\text{and}\;\; A^{+}=\{a_t:t\geq 0\}.
\end{equation*}
For $s>0$, set
\begin{equation*}
A^+_s:= \{a_t: 0\leq t \leq s\}\;\;
\text{ and }\;\;
A^-_s:=\{a_{-t}: 0\leq t \leq s\}.
\end{equation*}

Let $\mathit{E} \subset \mathbb{C}$ be a bounded Borel set.  Our goal in this section is to give an effective estimate of the following counting function:
\begin{equation*}
N_T(\mathcal{P},\mathit{E}):=\#\{C\in \mathcal{P}: C \cap \mathit{E} \neq \emptyset, \;\operatorname{Curv}(C) <T \},
\end{equation*}
where $\operatorname{Curv}(C)$ is the Euclidean curvature of $C$.

\subsection{Reformulation into orbit counting problem}

We built a relation between $N_T(\mathcal{P},\mathit{E})$
and a counting function of a $\Gamma$-orbit
 on $H\backslash G$ and prove Proposition \ref{reformulation 3}, which is a more
 precise version of Proposition 3.7 in \cite{Asymptotic}.
To obtain an effective estimate of $N_T(\mathcal{P},\mathit{E})$, it is important to understand the independence of $m_0$ from $\epsilon$
and the size of $T$ in terms of $\epsilon$, where $m_0$, $T$ and $\epsilon$ are described as Proposition \ref{reformulation 3}.

Fix a left invariant Riemannian metric on $G$, which induces
the hyperbolic metric on $G/K=\mathbb{H}^3$. 
 For any $\epsilon>0$, set $U_{\epsilon}$ to be the symmetric $\epsilon$-neighborhood of $e$ in $G$.  For any subset $W\subset G$, denote $W_{\epsilon}=U_{\epsilon}\cap W$. Define 
\begin{equation}
\label{neighborhood of a set}
 E^{+}_{\epsilon}:=\bigcup_{u\in U_\epsilon} uE \quad \text{and} \quad E^{-}_{\epsilon}:=\bigcap_{u\in U_{\epsilon}}uE.
\end{equation}
Note that there exists a constant $c_0>0$, independent of $E$,
 such that for any small $\epsilon>0$, $E^{+}_{\epsilon}$ contains the $c_0\epsilon$-neighborhood of $E$ in the Euclidean metric. We fix this constant $c_0$ in the following.

\begin{lem}
\label{reformulation 1}
For any $0<\epsilon <\frac{1}{c_0}$ and for any $T>\frac{1}{c_0\epsilon}$, we have 
\begin{equation*}
N_T(\mathcal{P},E) \leq \# \{C \in \mathcal{P}: \hat{C} \cap N_{E^{+}_{\epsilon}} A^{-}_{\log T}j \neq \emptyset\}.
\end{equation*} 
\end{lem}

\begin{proof}
It suffices to prove that if $C$ is a circle in $\mathcal{P}$ intersecting $E$  with $\operatorname{Curv}(C)<T$ for some $T>\frac{1}{c_0\epsilon}$,
then the intersection $\hat{C}\cap N_{E^{+}_{\epsilon}} A^{-}_{\log T}j$ is non-empty.

Observe that 
\begin{equation*}
N_{E^{+}_{\epsilon}} A^{-}_{\log T}j=\{z+lj: z\in E^{+}_{\epsilon}, T^{-1} \leq l \leq 1 \}.
\end{equation*}
Let $C\in \mathcal{P}$ be such that $C\cap E\ne \emptyset$ and $\operatorname{Curv}(C)<T$ for some $T>\frac{1}{c_0\epsilon}$.
Set $z_0$ and $r_C$ to be the center and the radius of $C$ respectively. 

If $r_C \leq c_0\epsilon$, then $z_0\in E^{+}_{\epsilon}$ and $z_0+r_Cj$ belongs to $\hat{C}\cap N_{E^{+}_{\epsilon}} A^{-}_{\log T}j$.

Now suppose $r_C>c_0\epsilon$. Choosing $z \in C \cap E$, set $w=(r_C-c_0\epsilon)\frac{z-z_0}{|z-z_0|}+z_0$. Then $w \in E^{+}_{\epsilon}$. Let $r>0$ be such that $w+rj$ lies in $\hat{C}$. We have
\begin{equation*} 
r=\sqrt{r_C^2-|w-z_0|^2}>T^{-1}.
\end{equation*}
 As a result, any point 
lying on the geodesic between $w+rj$ and $z$ and of Euclidean height greater than $T^{-1}$ and less than $1$ lies in $N_{E^{+}_{\epsilon}} A^{-}_{\log T}j$. This proves the claim.
\end{proof}

\begin{lem}
 \label{reformulation 2}
If $E$ is connected, then there exists a positive integer $m_0$, depending on $E$, such that for any  $T>0$,
\begin{equation*}
N_T(\mathcal{P},E) \geq \# \{C\in \mathcal{P}: \hat{C} \cap  N_{E} A^{-}_{\log T}j \neq \emptyset\}-m_0.
\end{equation*}
\end{lem}

\begin{proof} The local finiteness condition on $\mathcal{P}$ implies that
there are only finitely many lines in $\mathcal{P}$  intersecting $E$.
Denote by $W_T$ the set of all circles $C \in \mathcal{P}$ of positive curvature such that $\hat{C} \cap  N_{E} A^{-}_{\log T}j \neq \emptyset$ and $C \cap E =\emptyset$. It suffices to prove that $\cup_T W_T$ consists of finitely many circles. 

We claim  that any circle in $\cup_T W_T$ must have radius bigger than $d_E/2$, as well as intersect the $2/d_E$-neighborhood of $E$ where $d_E$ is the diameter of $E$.

If $C\in \cup_T W_T$, it follows from the connectedness of $E$  that $E$ is contained in the open disc enclosed by $C$, and hence the radius of $C$ is at least $d_E/2$.
Picking an arbitrary point $z+sj \in \hat{C}\cap N_{E} A^{-}_{\log T}j$, we let $w \in C$ be such that  $d(w,z)= d(C,z)$. Then
\begin{equation*}
d(w,z)=r_{C}-\sqrt{r_{C}^2-s^2} \leq {s^2}/{r_{C}}.
\end{equation*}
As mentioned above, $r_{C}>{d_E}/{2}$ and hence $d(w,z)\leq {2}/{d_{E}}$. Therefore $C$ intersects the ${2}/{d_E}$-neighborhood of $E$ non trivially, proving the claim.
Now the proposition follows from the assumption that $\mathcal{P}$ is locally finite.
\end{proof}

\begin{defn}[Definition of $B_T(E)$]
For $E \subset \mathbb{C}$ and $T>1$, define the subset $B_T(E)$ of $H\backslash G$ to be the image of the set
\begin{equation*}
KA^{+}_{\log T} N_{-E} =\{ka_tn_{-z}\in G: k\in K, 0 \leq t \leq \log T, z\in E \}
\end{equation*}
under the canonical projection $G \to H\backslash G$.
\end{defn}

Lemmas \ref{reformulation 1} and  \ref{reformulation 2} yield the following estimate:
\begin{prop}[cf. Proposition 3.7 in \cite{Asymptotic}]
\label{reformulation 3}
Suppose $E$   is connected. We have, for all small $0<\epsilon<\frac{1}{c_0}$ and for any $T>\frac{1}{c_0\epsilon}$, 
\begin{equation*}
\#([e]\Gamma \cap B_T(E)) -m_0 \leq N_T(\Gamma(C_0),E) \leq \# ([e]\Gamma \cap B_T(E^{+}_{\epsilon})),
\end{equation*}
where $m_0$ is a positive integer only depending on $E$.
\end{prop}

\subsection{Approximate $B_T(E)$ using $HAN$-decomposition}

We study the shape of
 $B_T(E)$ in an effective way. Note that there exists $c_1>0$ depending on the metric of $G$, such that for all small $\epsilon>0$, the set $K_{\epsilon}(0):=\{k\cdot 0:k\in K_{\epsilon}\}\subset \mathbb{C}$ contains the disk of radius $c_1\epsilon$ centered at $0$. Set $$T_{\epsilon}=-\log(c_1\epsilon).$$
 We show the following inclusion:
  \begin{prop}
 \label{B_T(E)}
There exists $c>0$ (independent of $E$) such that for all sufficiently small $\epsilon>0$, we have for all sufficiently large $T>1$
\begin{equation*}
\label{rewrite B_T(E)}
B_T(E)\subset \bigcup_{0\leq t\leq T_{\epsilon}}H\backslash Ha_tK(t)N_{-E}\,\cup\bigcup_{T_{\epsilon}-c\epsilon\leq t\leq \log T+c\epsilon}H\backslash Ha_tN_{-E^{+}_{c\epsilon}},
\end{equation*}
where $K(t)=\{k\in K:a_tk\in HKA^{+}\}$.
 \end{prop}

We first recall some results in \cite{Asymptotic}.

\begin{lem}[Proposition 4.2 and Corollary 4.3 in \cite{Asymptotic}]
\label{structure analysis 1}
\begin{enumerate}
\item If $a_t\in HKa_sK$ for $s>0$, then $|t|\leq s$. This in particular implies
\begin{equation*}
HKA^+_{\log T} =\bigcup_{0 \leq t \leq \log T}Ha_tK(t) \,\,\,\text{for any} \,\,T>1.
\end{equation*}
\item Given any small $\epsilon >0$, we have
\begin{equation*}
\{k \in K : a_tk \in HKA^+ \, \text{for  some}\, \,t >T_{\epsilon}\} \subset K_\epsilon M.
\end{equation*}
\end{enumerate}
\end{lem}
In fact, the second statement of the lemma is a more precise version of Proposition 4.2 (2) in \cite{Asymptotic}. We add to the original proof the observation that $K_{\epsilon}(0)$ contains the disk of radius $c_1\epsilon$ centered at $0$ for all small $\epsilon>0$.

\begin{lem}
\label{structure analysis 2}
There exist $c_2 >1$ and $t_0>1$, such that for any sufficiently small $\epsilon >0$,
\begin{equation*}
a_tkm_{\theta}\in Ha_tA_{c_2\epsilon}N_{c_2\epsilon}
\end{equation*}
holds for any $t>t_0$, $k\in K_{\epsilon}$ and $m_{\theta}\in M$.
\end{lem}

\begin{proof}
Fix $k\in K_{\epsilon}$ and $m_{\theta}\in M$. The product map $N^- \times A \times M \times N \to G$ is a diffeomorphism at a neighborhood of $e$, in particular, bi-Lipschitz. Hence there exists $l_1>1$ such that for all small $\epsilon>0$,
\begin{equation*}
K_\epsilon \subset N^-_{l_1\epsilon}A_{l_1\epsilon}M_{l_1\epsilon}N_{l_1\epsilon}.
\end{equation*}
So $k$ can be written as
\begin{equation*}
 k=n^-_1a_{s_1}m_{\theta_1}n_1,
 \end{equation*}
  where $n^-_1\in N^-_{l_1\epsilon}$, $a_{s_1} \in A_{l_1\epsilon}$, $m_{\theta_1} \in M_{l_1\epsilon}$ and $n_1\in N_{l_1\epsilon}$. Then
\begin{align}
\label{decompositon}
a_tk m_{\theta}&= a_tn^-_1a_{s_1}m_{\theta_1}n_1m
_{\theta} \\
&= (a_tn^-_1a_{-t})a_{t+s_1}m_{\theta_1}n_1m_{\theta}. \nonumber
\end{align}

Set $t_0$ to be a constant such that for all $t>t_0$, $a_tn^-_1a_{-t} \in N_\epsilon^{-}$. Due to the $H \times A \times N$ product decomposition of $G_\epsilon$, there exists $l_2>1$ such that for all small $\epsilon>0$, 
\begin{equation*}
N_{\epsilon}^{-} \subset H_{l_2\epsilon} A_{l_2\epsilon} N_{l_2\epsilon}.
\end{equation*}
This inclusion and (\ref{decompositon}) yield
\begin{align*}
a_tkm_{\theta}&= h_2a_{s_2}n_2 a_{t+s_1}m_{\theta_1}n_1m_{\theta} \\
&= h_2m_{\theta+\theta_1}a_{t+s_1+s_2}(m_{-\theta-\theta_1}a_{-(t+s_1)}n_2a_{t+s_1}m_{\theta+\theta_1})(m_{-\theta}n_1 m_{\theta})\nonumber\\
& \in  Ha_tA_{(l_1+l_2)\epsilon}N_{(l_1+l_2)\epsilon},\nonumber
\end{align*}
where $h_2 \in H_{l_2\epsilon}$, $a_{s_2}\in A_{l_2\epsilon}$ and $n_2\in N_{l_2\epsilon}$. Setting the constant $c_2=l_1+l_2$, we have
\begin{equation*}
a_tkm_{\theta}\in Ha_tA_{c_2\epsilon}N_{c_2\epsilon}.
\end{equation*}
\end{proof}

\begin{proof}[Proof of Proposition \ref{B_T(E)}]
For all sufficiently large $T>1$, it follows from Lemmas \ref{structure analysis 1} and \ref{structure analysis 2} that 
\begin{align*}
& B_T(E)\\
=&\bigcup_{0\leq t\leq  T_{\epsilon}}H\backslash Ha_tK(t)N_{-E} \cup \bigcup_{T_{\epsilon}\leq t\leq \log T}H\backslash Ha_tK(t)N_{-E} \\
\subset &\bigcup_{0\leq t\leq T_{\epsilon}}H\backslash Ha_tK(t)N_{-E}\cup \bigcup_{T_{\epsilon}\leq t\leq \log T}H\backslash Ha_tK_{\epsilon}MN_{-E} \\
\subset &\bigcup_{0\leq t\leq T_{\epsilon}}H\backslash Ha_tK(t)N_{-E}\cup \bigcup_{T_{\epsilon}-c_2\epsilon\leq t\leq \log T+c_2\epsilon}H\backslash Ha_tN_{c_2\epsilon}N_{-E}.
\end{align*}
\end{proof}

\subsection{On the measure $\omega_{\Gamma}$}
\begin{defn}
\label{measure on the complex plane}
Define a locally finite Borel measure $\omega_{\Gamma}$ on $\mathbb{C}$ as follows: fixing $x\in \mathbb{H}^3$, for $\psi \in C_{c}(\mathbb{C})$,
\begin{equation*}
\omega_\Gamma(\psi)=\int_{z \in \mathbb{C}} e^{\delta_\Gamma \beta_z(x,z+j)}\psi (z) d\nu_{\Gamma,x}(z).
\end{equation*}
\end{defn}
The definition of $\omega_{\Gamma}$ is independent of the choice of $x\in \mathbb{H}^3$ by the conformal properties of $\nu_{\Gamma,x}$. 

\begin{lem}[Lemma 5.2 in \cite{Asymptotic}]
\label{Busemann function}
For any $x=p+rj\in \mathbb{H}^3$, and $z\in  \mathbb{C}$, we have
\begin{equation*}
\beta_{z}(p+rj,z+j)=\log \frac{|z-p|^2+r^2}{r}.
\end{equation*}
\end{lem}
This lemma provides another formula for $\omega_{\Gamma}$: for any $\psi\in C_c(\mathbb{C})$, 
\begin{equation*}
\omega_{\Gamma}(\psi)=\int_{z\in \mathbb{C}} (|z|^2+1)^{\delta} \psi(z)d\nu_{j}(z).
\end{equation*}



\subsubsection {Relation between $\omega_{\Gamma}$ and $m^{\operatorname{BR}}$}

For a bounded  Borel set $E \subset \mathbb{C}$, let $E^{+}_{\epsilon}$ and $E^{-}_{\epsilon}$ be the sets defined as (\ref{neighborhood of a set}). For small $\epsilon>0$, let $\psi^{\epsilon}$ be a non-negative smooth function in $C(G)$ supported in $U_{\epsilon}$ with integral one. Set $\Psi^{\epsilon} \in C^{\infty}_{c}(\Gamma \backslash G)$ to be the $\Gamma$-average of $\psi^{\epsilon}$:
\begin{equation*}
\Psi^{\epsilon}(\Gamma g):=\sum_{\gamma \in \Gamma} \psi^{\epsilon}(\gamma g).
\end{equation*} 
For a bounded Borel subset $E\subset \mathbb{C}$, let $$h_E=\max_{z\in E}{|z|}+1.$$

\begin{prop}[cf. Lemma 5.7 in \cite{Asymptotic}]
\label{relation of measures}
There exists $c>0$ independent of $E$ such that for all small $\epsilon>0$,
\begin{equation*}
(1-c\cdot h_{E}\cdot \epsilon) \cdot \omega_{\Gamma}(E^{-}_{\epsilon})\leq \int_{z\in E}m^{\operatorname{BR}}(\Psi^{\epsilon}_{-z})dn_{-z} \leq (1+c\cdot h_{E} \cdot \epsilon) \cdot \omega_{\Gamma}(E^{+}_{\epsilon}),
\end{equation*}
where $\Psi^{\epsilon}_{-z}\in C^{\infty}_{c}(\Gamma \backslash G)^{M}$ is given by $\Psi^{\epsilon}_{-z}(g)=\int_{m\in M}\Psi^{\epsilon}(gmn_{-z})dm$ with $dm$ the probability measure on $M$.
\end{prop}

\begin{lem}[Lemma 5.5 in \cite{Asymptotic}]
\label{coordinates for product}
If $(m_{\theta}a_tn_w^{-}n_z)(m_{\theta_1}a_{t_1}n_{w_1}^{-}n_{z_1})=m_{\theta_0}a_{t_0}n_{w_0}^{-}n_{z_0}$ in the $MAN^{-}N$ coordinates, then
\begin{align*}
t_0 &=t+t_1+2\log(|1+e^{-t_1-2i\theta_1}w_1z|),\\
z_0 &=e^{-t_1-2i\theta_1}z(1+e^{-t_1-2i\theta_1}w_1z)^{-1}+z_1.
\end{align*}
\end{lem}

\begin{proof}[Proof of Proposition \ref{relation of measures}]
Consider the following function on $MAN^{-}N\subset G$:
\begin{equation*}
\mathcal{R}_{E}(ma_tn_w^{-}n_z)=e^{-\delta t}\chi_{E}(-z).
\end{equation*}
We may regard $\mathcal{R}_E$ as a function defined on $G$. It is shown in Lemma 5.7 in \cite{Asymptotic} that
\begin{equation*}
\int_{z\in E}m^{\operatorname{BR}}(\Psi^{\epsilon}_{-z})dn_{-z}=\int_{g\in U_{\epsilon}}\psi^{\epsilon}(g)\int_{k\in K}\mathcal{R}_E(k^{-1}g)d\nu_j(k(0))dg.
\end{equation*}
Write $k^{-1}=m_{\theta}a_{t}n_w^{-}n_z$ and $g=m_{\theta_1}a_{t_1}n_{w_1}^{-}n_{z_1}\in U_{\epsilon}$. By Lemma \ref{coordinates for product}, we have $k^{-1}g=m_{\theta_0}a_{t_0}n_{w_0}^{-1}n_{z_0}$ with $t_0=t+t_1+2\log(|1+e^{-t_1-2i\theta_1}w_1z|)$. Noting that $\mathcal{R}_E(k^{-1}g)=e^{-\delta t_0}\chi_E(g^{-1}k(0))$, we have for any $g\in U_{\epsilon}$
\begin{equation*}
\int_{k\in K/M}\mathcal{R}_E(k^{-1}g)d\nu_j(k(0))=(1+O(h_E\cdot \epsilon))\int_{k(0)\in E^{\pm}_{\epsilon}}e^{-\delta t}d\nu_j(k(0)).
\end{equation*}
Using Proposition 5.4 in \cite{Asymptotic}, we have
\begin{equation*}
\int_{k(0)\in E^{\pm}_{\epsilon}}e^{-\delta t}d\nu_j(k(0))=\omega_{\Gamma}(E^{\pm}_{\epsilon}),
\end{equation*}
which yields the proposition.
\end{proof}

 \subsubsection {Regularity criterion for $\omega_{\Gamma}$}
 
 Fix a bounded Borel set $E\subset \mathbb{C}$. To apply the effective result in our paper, it is important to understand the difference between $\omega_{\Gamma}(E^{\pm}_{\epsilon})$ and $\omega_{\Gamma}(E)$.

\begin{defn}[Regularity condition]
\label{regularity condition}
We call a bounded Borel subset $E \subset \mathbb{C}$ regular if there exists $0<p<1$ such that for all small $\epsilon>0$,
\begin{equation}
\omega_{\Gamma}(E^{+}_{\epsilon}-E^{-}_{\epsilon}) =O(\epsilon^p),
\end{equation}
where the implied constant depends only on $E$.
\end{defn}

\begin{prop}[\cite{StraUrb}]
Suppose  $\Gamma$ is convex cocompact and Zariski dense. For any bounded set $E\subset \mathbb{C}$, if $\partial E$ is a finite union of proper subvarieties, then $E$ is regular in the above sense.
\end{prop}
 
Denote by $\Lambda_{p}(\Gamma)$  the set of parabolic limit points. For $\xi\in\Lambda_p(\Gamma)$, the rank of $\xi$ is the rank of the abelian subgroup of $\Gamma$ which fixes $\xi$. As $\Gamma\subset \operatorname{PSL}_2(\mathbb{C})$, then $\operatorname{rank}(\xi)$ is either 1 or 2. We provide a regularity criterion for $\Gamma$ with $\Lambda_{p}(\Gamma)\neq \emptyset$.

\begin{prop} Let $k_0:=\max_{\xi\in \Lambda_p(\Gamma)}\operatorname{rank}(\xi)$. If the critical exponent of $\Gamma$ satisfies
\begin{equation*}
\delta> \max \{1,\frac{k_0+1}{2}\},
\end{equation*}
then any bounded Borel set $E \subset \mathbb{C}$ with $\partial E$ a finite union of rectifiable curves is regular.
\end{prop}

\begin{proof}
The proof is adapted from the proof of Proposition 7.10 in \cite{Matrix coefficients}. For $\xi\in \partial \mathbb{H}^3$, denote by  $s_{\xi}=\{\xi_t:t\in [0,\infty)\}$ the geodesic ray emanating from $j$ toward $\xi$ and let $S(\xi_t) \in \mathbb{H}^3$ be the unique 2-dimensional geodesic plane which is orthogonal to $s_{\xi}$ at the point $\xi_t$. Denote by $b(\xi_t)$ the projection from $j$ onto $\partial \mathbb{H}^3$ of $S(\xi_t)$, that is, 
\begin{equation*}
b(\xi_t)=\{\xi' \in \partial \mathbb{H}^3: s_{\xi'} \cap S(\xi_t) \neq \emptyset\}.
\end{equation*}

It is shown in \cite{Stratmann} and \cite{Sullivan} that there exists a $\Gamma$-invariant collection of pairwise disjoint horoballs $\{\mathcal{H}_{\xi}: \xi \in \Lambda_p(\Gamma)\}$ for which the following holds: there exists a constant $c>1$ such that for any $\xi \in \Lambda(\Gamma)$ and for any $t>0$, 
\begin{equation*}
c^{-1}e^{-\delta t}e^{d(\xi_t,\Gamma(j))(k(\xi_t)-\delta)} \leq \nu_j(b(\xi_t)) \leq c e^{-\delta t}e^{d(\xi_t,\Gamma(j))(k(\xi_t)-\delta)},
\end{equation*}
where $k(\xi_t)$ is the rank of $\xi'$ if $\xi_t \in \mathcal{H}_{\xi'}$ for some $\xi' \in \Lambda_p(\Gamma)$ and $\delta$ otherwise. Using $0 \leq d(\xi_t,\Gamma(j)) \leq t$, we have for any $\xi \in \Lambda(\Gamma)$ and $t>1$,
\begin{equation}
\label{shadow inequality}
\nu_j(b(\xi_t)) \ll 
\begin{cases}
e^{(-2\delta+k(\xi_t))t} \quad \mbox{if $k(\xi_t) \geq \delta$} \\
e^{-\delta t} \quad \mbox{otherwise} .
\end{cases}
\end{equation}

By standard computation in hyperbolic geometry, for any bounded set $E \subset \mathbb{C}$, there exist $c'>1$ and $0<r_0<1$ such that for any $t>-\log r_0$ and $\xi\in E$,
\begin{equation*}
B(\xi,  e^{-t}/c') \subset b(\xi_t) \subset B(\xi, c' e^{-t})
\end{equation*}
where $B(\xi,r)$ is the  closed Euclidean ball in $\partial \mathbb{H}^3$ of radius $r$. Since $d \omega_{\Gamma}=(|z|^2+1)^{\delta_{\Gamma}} d \nu_j$, setting $k_0:=\max_{\xi' \in \Lambda_p(\Gamma)} \text{rank}(\xi')$, it follows from (\ref{shadow inequality}) that for all small $\epsilon >0$ and $\xi \in \Lambda(\Gamma) \cap E$
\begin{equation*}
\omega_{\Gamma}(B(\xi,\epsilon)) \ll \nu_j(B(\xi,\epsilon)) \ll \epsilon ^{\delta}+ \epsilon^{2\delta-k_0}.
\end{equation*}

Suppose $\partial E$ is a finite union of rectifiable curves. The collection
\begin{equation*} 
\mathcal{F}=\{ B(\xi,\epsilon):\xi\in \partial E\}
\end{equation*}
 forms a Besicovitch covering for $\partial E$. By the Besicovitch covering theorem, there exist $d\in \mathbb{N}$ (independent of $\epsilon$ and $E$) and  a countable subcollection of $\mathcal{F}$, say $\{B(\xi_n,\epsilon)\}$, such that 
 \begin{equation*}
 \partial E\subset \bigcup B(\xi_n,\epsilon),
 \end{equation*}
 and  the multiplicity of the cover $\bigcup B(\xi_n,\epsilon)$ is at most $d$. Since $\partial E$ is a finite union of rectifiable curves, we have
\begin{equation}
\label{length}
\sum l(B(\xi_n,\epsilon)\cap \partial E)\ll  l(\partial E)<\infty,
\end{equation} 
where $l(\cdot)$ denotes the length function. This implies 
\begin{equation*}
\# \{B(\xi_n,\epsilon)\} =O(\epsilon^{-1}).
\end{equation*}

Now cover the $\epsilon$-neighborhood of $\partial E$ by $\{B(\xi_n,2\epsilon)\}$. Note that $\nu_j$ is supported on $\Lambda(\Gamma)$. For every $B(\xi_n, 2\epsilon)$, if $\xi_n\notin \Lambda(\Gamma)$ but $B(\xi_n, 2\epsilon)$ intersects $\Lambda(\Gamma)$ nontrivially, replace it by a ball of radius $4\epsilon$ with center in $\Lambda(\Gamma)$. We obtain the estimate
\begin{equation}
\omega_{\Gamma}(\epsilon\text{-nbhd of} \,\, \partial E )=O(\epsilon ^{\delta-1}+\epsilon ^{2\delta-k_0-1}),
\end{equation}
where the constant depending only on $E$. Therefore, if $\delta > \max \{1,\frac{k_0+1}{2}\}$,  then $E$ is regular.
\end{proof}

\subsection{Conclusion}
We keep the notations from Section 2. One of our main theorems is the following:
\begin{thm}
\label{effective circle counting thm}
Assume $\Gamma$ is either convex cocompact or its critical exponent $\delta$ is greater than $1$. Let $\mathcal{P}$ be a locally finite circle packing in $\hat{\mathbb{C}}$ invariant under $\Gamma$ with finitely many $\Gamma$-orbits. For any bounded connected regular set $E$ in $\mathbb{C}$, there exists $\eta>0$, such that as $T \to \infty$, we have
\begin{equation*}
N_T(\mathcal{P},E)=\frac{\operatorname{sk}_{\Gamma}(\mathcal{P})}{\delta \, |m^{\operatorname{BMS}}|}\,T^{\delta}\, \omega_{\Gamma}(E)+O(T^{\delta-\eta}).
\end{equation*}
\end{thm}
The rest of the section is devoted to the proof of Theorem \ref{effective circle counting thm}.

For $\psi \in C^{\infty}(\Gamma \backslash G)$ and $l\in \mathbb{N}$, we consider the following $L^2$-Sobolev norm of degree $l$:
\begin{equation*}
S_{l}(\psi)=\sum \lVert X(\psi)\rVert_{2},
\end{equation*}
where the sum is taken over all monomials $X$ in a fixed basis of the Lie algebra of $G$ of order at most $l$ and $\lVert X(\psi)\rVert_{2}$ is the $L^2(\Gamma \backslash G)$-norm of $X(\psi)$. For $\varphi\in C^{\infty}(\Gamma_{H}\backslash H)$, $S_l(\varphi)$ is defined similarly.

The key ingredient in the proof of Theorem \ref{effective circle counting thm} is the following effective equidistribution result:
\begin{thm} [\cite{Matrix coefficients}]
\label{effective distribution}
Assume $\Gamma$ is convex cocompact or its critical exponent $\delta$ is greater than $1$. Suppose the natural projection $\Gamma_{H}  \backslash C_{0}^{\dagger} \to \Gamma \backslash G/M$ is proper. Then there exist $\eta_0 >0$ (depending on the spectral gap data for $\Gamma$) and $l\in \mathbb{N}$ such that for any compact subset $\Omega \subset \Gamma \backslash G/M$, any $\Psi \in C^{\infty} (\Omega)$ and any bounded $\phi \in C^{\infty}(\Gamma_{H} \backslash C_0^{\dagger})$, as $t \to \infty$,

\begin{equation*}
e^{(2-\delta)t} \int _{h\in \Gamma_{H} \backslash C^{\dagger}_{0}}\Psi (ha_t)\phi (h)dh
 =\frac{\mu^{\operatorname{PS}}_{H}(\phi)}{|m^{\operatorname{BMS}}|}m^{\operatorname{BR}}(\Psi)+O(S_l(\Psi)\cdot S_l(\phi)e^{-\eta_0 t}),
\end{equation*}
where the implied constant depends only on $\Omega$.
\end{thm}

\begin{rem}
\begin{enumerate}
\item If $\delta>1$, then $\Gamma$ is Zariski dense (Lemma 2.11 in \cite{Ergodicity of Kleinian groups}).


\item  Strictly speaking, Theorem \ref{effective distribution} in \cite{Matrix coefficients} is shown using the exponential mixing of geodesic flow on $T^1(\Gamma\backslash \mathbb{H}^3)$. Such an exponential mixing is provided in \cite{Stoyanov} and \cite{PetSto} for $\Gamma$ convex cocompact and in \cite{Matrix coefficients} for $\Gamma$ with critical exponent greater than 1.

\end{enumerate}
\end{rem} 


For every $T>1$, and $0<\epsilon<1$, denote
\begin{align*}
V^{+}_{T,\epsilon}(E^{+}_{2\epsilon})&:=\{a_sn_{-z}:z\in E^{+}_{2\epsilon},\,\,-\epsilon\leq s\leq \log T+\epsilon\},\\
V^{-}_{T,\epsilon}(E^{-}_{2\epsilon})&:=\{a_sn_{-z}:z\in E^{-}_{2\epsilon},\,\epsilon \leq s\leq \log T-\epsilon\}.
\end{align*}

\begin{lem}[Lemma 6.3 in \cite{Asymptotic}]
\label{stability of KAN-decomposition}
There exists $c_3>0$, such that for all $T>1$ and for all small $\epsilon>0$,
\begin{align*}
&KA^{+}_{\log T} N_{-E^{+}_{\epsilon}}U_{c_3\epsilon} \subset KV^{+}_{T,\epsilon}(E^{+}_{2\epsilon}),\\
&KV^{-}_{T,\epsilon}(E^{-}_{2\epsilon})U_{c_3\epsilon} \subset KA^{+}_{\log T}N_{-E}.
\end{align*}
\end{lem}
For simplicity, denote the subsets:
\begin{equation}
\label{two sets}
W^{+}_{T,\epsilon}:=H\backslash HKV^{+}_{T,\epsilon}(E^{+}_{2\epsilon})\,\,\,\text{and}\,\,\, W^{-}_{T,\epsilon}:=H\backslash HKV^{-}_{T,\epsilon}(E^{-}_{2\epsilon}).
\end{equation}
Define the counting functions $F^{\epsilon,\pm}_{T}$ on $\Gamma \backslash G$:
\begin{equation*}
F^{\epsilon,+}_{T}(g) :=\sum_{\gamma \in \Gamma_{H}\backslash \Gamma} \chi_{W^{+}_{T,\epsilon}}([e]\gamma g)\,\,\,\text{and}\,\,\,
F^{\epsilon,-}_{T}(g) :=\sum_{\gamma \in \Gamma_{H}\backslash \Gamma} \chi_{W^{-}_{T,\epsilon}}([e]\gamma g).
\end{equation*}
Let $c_3$ be as Lemma \ref{stability of KAN-decomposition}. The following lemma can  easily be deduced from Proposition \ref{reformulation 3} and Lemma \ref{stability of KAN-decomposition}.
\begin{lem}[cf. Lemma 6.4 in \cite{Asymptotic}]
Given any small $\epsilon>0$, we have, for all $g\in U_{c_3\epsilon}$ and $T>\frac{1}{c_0\epsilon}$
\begin{equation}
\label{counting inequality}
F^{\epsilon,-}_{T}(g)-m_0 \leq N_T(\Gamma(C_0),E) \leq F^{\epsilon,+}_{T}(g),
\end{equation}
where $m_0$ is a positive integer depending only on $E$.
\end{lem}

\begin{proof}[\textbf{Proof of Theorem \ref{effective circle counting thm}}]
\textbf{Step 1:} We first prove the theorem for $\mathcal{P}=\Gamma (C_0)$.

For $\epsilon>0$, let $\psi^{\epsilon}$ be a non-negative function in $C^{\infty}_{c}(G)$ supported in $U_{c_3\epsilon}$ with integral one. Denote by $\Psi^{\epsilon}\in C^{\infty}_{c}(\Gamma \backslash G)$ the $\Gamma$-average of $\psi^{\epsilon}$. For $T>\frac{1}{c_0\epsilon}$, integrating (\ref{counting inequality}) against $\Psi^{\epsilon}$, we obtain
\begin{equation*}
\langle F^{\epsilon,-}_{T}, \Psi^{\epsilon}\rangle-m_0 \leq N_T(\Gamma(C_0),E) \leq \langle F^{\epsilon,+}_{T}, \Psi^{\epsilon}\rangle.
\end{equation*}

By abusing notation, we use $dh$ to denote the Haar measures on $H$ and $H/M$. We require that these two measures are compatible with the probability measure $dm$ on $M$. The following defines a Haar measure on $G$: for $g=ha_rk\in HA^{+}K$,
\begin{equation*}
dg=4\sinh r\cdot \cosh r dh dr dm_j(k),
\end{equation*}
where $dm_j(k):=dm_j(kX_0^{+})$. Denote by $d\lambda$ the unique $G$-invariant measure on $H\backslash G$ which is compatible with $dg$ and $dh$.

For $\langle F^{\epsilon,+}_{T},\Psi^{\epsilon}\rangle$, we have
\begin{align}
\label{counting equation}
\langle F^{\epsilon,+}_{T}, \Psi^{\epsilon}\rangle &=\int_{\Gamma \backslash G} \sum_{\gamma \in \Gamma_{H}\backslash \Gamma} \chi_{W^{+}_{T,\epsilon}}([e]\gamma g)\Psi^{\epsilon}(g) dg\\
&=\int_{g\in \Gamma_{H}\backslash G}  \chi_{W^{+}_{T,\epsilon}}([e] g)\Psi^{\epsilon}(g)dg \nonumber\\
&= \int_{g \in W^{+}_{T,\epsilon}} \int_{h \in \Gamma_H \backslash C^{\dagger}_0} \int_{m\in M} \Psi^{\epsilon}(hmg) dm dh d\lambda(g)\nonumber.
\end{align}

Consider the set $W^{+}_{T,\epsilon}$ defined in (\ref{two sets}). We can rewrite it in the following form by  Lemma \ref{structure analysis 1} (1)
\begin{equation*}
W^{+}_{T,\epsilon}=\bigcup_{0\leq s\leq \log T+\epsilon}H\backslash Ha_sK(s)N_{-E^{+}_{2\epsilon}},
\end{equation*}
where the set $K(\cdot)$ is defined as Proposition \ref{rewrite B_T(E)}. Applying Proposition \ref{rewrite B_T(E)}, we get:
\begin{align}
\label{decomposition by t}
W^{+}_{T,\epsilon}&\subset \bigcup_{0\leq t\leq T_{\epsilon}}H\backslash Ha_tK(t)N_{-E^{+}_{2\epsilon}}\,\cup\bigcup_{T_{\epsilon}-\rho_1\epsilon\leq t\leq \log T+\rho_1\epsilon}H\backslash Ha_tN_{-E^{+}_{\rho_1\epsilon}}\\
&:=V_1\cup V_2,\nonumber
\end{align}
where $T_{\epsilon}=-\log (c_1\epsilon)$ and $\rho_1>0$ is some constant.

 Notice that the measure $e^{2t}dtdn$ is a right invariant measure of $AN$ and $[e]AN$ is an open subset in $H\backslash G$. Hence $d \lambda (a_tn)$ (restricted to $[e]AN$) and $e^{2t}dtdn$ are constant multiplies of each other. It follows from the formula of $dg$ that $d\lambda (a_tn)=e^{2t}dt dn$. Besides, observe that the local finiteness of $\Gamma(C_0)$ implies the map $\Gamma_H\backslash C_0^{\dagger}\to \Gamma\backslash G/M$ is proper. Using Theorem \ref{effective distribution}, we have
\begin{align}
\label{integral 8}
&\int_{V_2}\int_{\Gamma_{H}\backslash C^{\dagger}_0} \int_{M} \Psi^{\epsilon}(hmg)dm dh d\lambda (g)\\
= & \int_{z \in E^{+}_{\rho_1\epsilon}} \int_{T_{\epsilon}-\rho_1\epsilon}^{\log T+\rho_1\epsilon} e^{2s} \int_{\Gamma_{H}\backslash C^{\dagger}_0} \int_{M} \Psi^{\epsilon} (hma_sn_{-z})dm dh ds dn_{-z} \nonumber\\
\leq & \int_{z \in E^{+}_{\rho_1\epsilon}} \int_{T_{\epsilon}-\rho_1\epsilon}^{\log T+\rho_1\epsilon} \frac{|\mu^{\operatorname{PS}}_{H}|}{|m^{\operatorname{BMS}}|} m^{\operatorname{BR}}(\Psi^{\epsilon}_{-z})e^{\delta s} +\rho_2\cdot(S_l(\Psi^{\epsilon}_{-z})e^{(\delta-\eta_0)s}) ds dn_{-z} \nonumber\\
\leq & \frac{|\mu^{\operatorname{PS}}_{H}|}{\delta \,|m^{\operatorname{BMS}}|}T^{\delta}(1+\rho_3\cdot\epsilon) \,\int_{z \in E^{+}_{\rho_1\epsilon}} m^{\operatorname{BR}}(\Psi^{\epsilon}_{-z}) dn_{-z}\, +\rho_3\cdot(S_l(\Psi^{\epsilon})T^{\delta-\eta_0}), \nonumber
\end{align}
where $\Psi^{\epsilon}_{-z}(g)=\int_{m\in M}\Psi^{\epsilon}(gmn_{-z})dm$ and $\rho_2,\,\rho_3>0$ are some constants. To obtain the last inequality above, we use the estimate that $\sup_{z\in E^{+}_{\rho_1\epsilon}} S_l(\Psi^{\epsilon}_{-z})\ll S_l(\Psi^{\epsilon})$ because for any monomial $X$ of order 1, we have $X(\Psi^{\epsilon}_{-z})(g)=\int_{m\in M}\operatorname{Ad}_{n_zm}X(\Psi^{\epsilon})(gmn_{-z})dm$.

Since $E$ is bounded and regular, Proposition \ref{relation of measures} implies
\begin{align*}
\int_{z \in E^{+}_{\rho_1\epsilon}} m^{\operatorname{BR}}(\Psi^{\epsilon}_{-z})dn_{-z} 
\leq & (1+\rho_4\cdot\epsilon)\cdot \omega_{\Gamma} (E^{+}_{\rho_4\epsilon})\\
\leq & (1+\rho_5\cdot\epsilon^p)\,\omega_{\Gamma}(E),
\end{align*}
 where  $\rho_4,\,\rho_5>0$ are some constants and $p$ is the constant appearing in the definition of  the regularity of $E$ (Definition \ref{regularity condition}).

Therefore,
\begin{align}
\label{integral 9}
&\int_{V_2}\int_{\Gamma_{H}\backslash C_{0}^{\dagger}} \int_{M} \Psi^{\epsilon}(hmg)dm dh d\lambda(g)\\
\leq & \frac{|\mu^{\operatorname{PS}}_{H}|}{\delta \,|m^{\operatorname{BMS}}|}T^{\delta} \omega_{\Gamma} (E)+ \rho_6\cdot(\epsilon^{p}T^{\delta}+\epsilon^{-(3+l)}T^{\delta-\eta_0}) && (\text{for some}\,\rho_6>0),\nonumber
\end{align}
where we use the estimate $S_l(\Psi^{\epsilon})=O(\epsilon^{-(3+l)})$ since dim $G$=6.

Consider $V_1$ in (\ref{decomposition by t}). Fix small $\epsilon'$ such that $\epsilon'>\epsilon$ and it satisfies Proposition \ref{rewrite B_T(E)}.  We decompose $V_1$ into two parts using Proposition \ref{rewrite B_T(E)}:
\begin{align*}
V_1 &\subset \bigcup_{0\leq t \leq T_1} H\backslash Ha_tK(t)N_{-E^{+}_{2\epsilon}} \cup \bigcup_{T_1-\rho_7\epsilon'\leq t\leq T_{\epsilon}+\rho_7\epsilon'} H\backslash Ha_tN_{-E^{+}_{\rho_7\epsilon'}}\\
&=: V_3 \cup V_4,
\end{align*}
where $T_1:=-\log (c_1 \epsilon')$ and $\rho_7>0$ is some constant.

For $V_4$, by the similar way as we get (\ref{integral 9}), we have
\begin{align*}
&\int_{V_4} \int_{\Gamma_{H} \backslash C_0^{\dagger}} \int_{M} \Psi^{\epsilon}(hmg) dm dh d\lambda(g)\\
\leq & 8\, e^{\delta T_{\epsilon}} \frac{|\mu^{\operatorname{PS}}_{H}|}{\delta \,|m^{\operatorname{BMS}}|} \omega_{\Gamma} (E^{+}_{\rho_{8}\epsilon'}),
\end{align*}
for some constant $\rho_{8}>0$.

Since $E$ is bounded and $\epsilon'$ is fixed, $\omega_{\Gamma} (E^{+}_{\rho_{8}\epsilon'})=O(1)$. As $T_{\epsilon}=-\log(c_1\epsilon)$, we get
\begin{align}
\label{integral 10}
\int_{V_4} \int_{\Gamma_{H} \backslash C_0^{\dagger}} \int_{M} \Psi^{\epsilon}(hmg) dm dh d\lambda(g) =O(\epsilon^{-\delta}).
\end{align}

For $V_3$, reversing the process of translating the circle counting into orbit counting, we have
\begin{align}
\label{integral 11}
&\int_{V_3} \int_{\Gamma_{H}\backslash C_0^{\dagger}} \int_{M} \Psi^{\epsilon}(hmg)dm dh d\lambda(g)\\
=& O(\#\{C\in \Gamma (C_0): C\cap E^{+}_{2\epsilon}\neq \emptyset, \operatorname{Curv}(C)\leq T_1\}) \nonumber\\
=& O(1),\nonumber
\end{align}
as $T_1$ is fixed. 

Adding (\ref{integral 9}), (\ref{integral 10}) and (\ref{integral 11}) together, we get
\begin{align*}
\langle F^{\epsilon,+}_{T}, \Psi^{\epsilon}\rangle \leq \frac{|\mu^{\operatorname{PS}}_{H}|}{\delta \,|m^{\operatorname{BMS}}|} T^{\delta} \omega_{\Gamma}(E)+\rho_{9}\cdot(\epsilon^p T^{\delta}+\epsilon^{-(3+l)}T^{\delta-\eta_0}+\epsilon^{-\delta}),
\end{align*}
for some constant $\rho_{9}>0$.

Set $\eta'=\min\{\frac{\delta p}{\delta +p}, \frac{\eta_0 p}{3+l+p}\}$ and $\epsilon=\max\{T^{\frac{\eta'}{\delta}-1}, T^{\frac{\eta'-\eta_0}{3+l}}\}$. Then $T>\frac{1}{c_0\epsilon}$. Hence 
\begin{equation*}
N_T(\Gamma (C_0), E) \leq \langle F^{\epsilon,+}_{T}, \Psi^{\epsilon}\rangle \leq \frac{|\mu^{\operatorname{PS}}_{H}|}{\delta \, |m^{\operatorname{BMS}}|} T^{\delta}\omega_{\Gamma}(E)+\rho_{9}\cdot T^{\delta-\eta'}.
\end{equation*}

For $\langle F^{\epsilon,-}_{T}, \Psi^{\epsilon}\rangle$,  the definition of $W^{-}_{T,\epsilon}$ (\ref{two sets}) implies the following inclusion:
\begin{equation*}
 W^{-}_{T,\epsilon}\supset \bigcup_{T_2\leq t \leq \log T-\epsilon}H\backslash Ha_tN_{-E^{-}_{2\epsilon}},
 \end{equation*}
 where $T_2$ is some large fixed number. 
 
 Using similar argument as above, we have
\begin{equation}
\langle F^{\epsilon,-}_{T}, \Psi^{\epsilon}\rangle \geq \frac{|\mu^{\operatorname{PS}}_{H}|}{\delta \, |m^{\operatorname{BMS}}|} T^{\delta}\omega_{\Gamma}(E)+\rho_{10} \cdot T^{\delta-\frac{\eta_{0} p}{3+l+p}},
\end{equation}
for some constant $\rho_{10}>0$.

Therefore, there exists $\eta>0$ so that as $T\to \infty$, we have
\begin{equation}
N_T(\Gamma(C_0),E)=\frac{|\mu^{\operatorname{PS}}_{H}|}{\delta \, |m^{\operatorname{BMS}}|} \,T^{\delta}\,\omega_{\Gamma}(E)+O (T^{\delta-\eta}).
\end{equation}

\textbf{Step 2:} We prove the theorem for a general circle packing.

Let $C$ be any circle in the circle packing $\mathcal{P}$. Denote the radius of $C$ by $r$ and the center of $C$ by $z_0$. Set
\begin{equation*}
g_C:=n_{z_0}a_{\log r}=\begin{pmatrix} 1 & z_0\\ 0 & 1\end{pmatrix} \begin{pmatrix}  \sqrt{r} & 0\\ 0 & \sqrt{r^{-1}}\end{pmatrix}.
\end{equation*}
Then $g_C(C_0)=C$. And 
\begin{align*}
N_T(\Gamma(C),E) &= N_{rT}(g_C^{-1}\Gamma g_C(C_0), g_C^{-1}(E)).
\end{align*}

\begin{lem}[Lemma 6.5 in \cite{Asymptotic}]
\label{total mass of BMS}
\begin{equation*}
|m^{\operatorname{BMS}}_{g_C^{-1}\Gamma g_C}|=|m^{\operatorname{BMS}}_{\Gamma}|.
\end{equation*}
\end{lem}

\begin{lem}[Lemma 6.7 in \cite{Asymptotic}]
\label{relate measures on complex plane}
For any bounded Borel subset $E\subset \mathbb{C}$,
\begin{equation*}
\omega_{g_C^{-1}\Gamma g_C}(g_C^{-1}(E))=r^{-\delta}\omega_{\Gamma}(E).
\end{equation*}
\end{lem}
Using the above two lemmas, we obtain
\begin{align}
\label{conclusion}
 N_{T}(\Gamma(C),E)
 =&\frac{|\mu^{\operatorname{PS}}_{g_C^{-1}\Gamma g_C,H}|}{\delta \cdot |m^{\operatorname{BMS}}_{g_C^{-1}\Gamma g_C}|} (rT)^{\delta}\omega_{g_C^{-1}\Gamma g_C}(g_C^{-1}(E))+O (T^{\delta-\eta})\\
=&\frac{|\mu^{\operatorname{PS}}_{g_C^{-1}\Gamma g_C,H}|}{\delta \cdot |m^{\operatorname{BMS}}_{\Gamma}|} T^{\delta}\omega_{\Gamma}(E)+O (T^{\delta-\eta}).\nonumber
\end{align}
Since $\mathcal{P}$ consists of finitely many $\Gamma$-orbits,  (\ref{conclusion}) implies our theorem.
\end{proof}

\section{Effective circle count for circle packing in  ideal triangle of $\mathbb{H}^2$}

We use the upper half plane model for $\mathbb{H}^2$:
\begin{equation*}
\mathbb{H}^2=\{z\in \mathbb{C}:\operatorname{Im}z>0\}.
\end{equation*} 

For any circle packing $\mathcal{P}$ contained in $\mathbb{H}^2$ and any $t>0$, we define the following counting function:
\begin{equation*}
N_t(\mathcal{P}):=\#\{C \in \mathcal{P}: \operatorname{Area}_{\operatorname{hyp}}(C)>t\},
\end{equation*}
where $\operatorname{Area}_{\operatorname{hyp}} (C)$ is the hyperbolic area of the disk enclosed by $C$.

Let $\mathcal{T}$ be the ideal triangle in $\mathbb{H}^2$ whose three sides are given by the circles $\{x=\pm1\}$ and $\{z\in \mathbb{C}:|z|=1\}$. Let $\mathcal{P}(\mathcal{T})$ be the circle packing attained by filling in the largest inner circles. This section is devoted to give an effective estimate of $N_t(\mathcal{P}(\mathcal{T}))$.

Denote by $\mathcal{P}_0$ the Apollonian circle packing generated by the circles $\{x=\pm 1\}$, $\{\lvert z\rvert=1\}$ and $\{|z-2i|=1\}$. It is clear that $\mathcal{P}(\mathcal{T})$ is a part of $\mathcal{P}_0$ (see Figure 3). Let $C_1$, $C_2$, $C_3$ and $C_4$ be the circles  $\{|z-(1+i)|=1\}$, $\{|z-(-1+i)|=1\}$, $\{y=0\}$ and  $\{y=2\}$ respectively. Set $S_i$ to be the reflection with respect to $C_i$. Denote $\operatorname{PSL}_2(\mathbb{C})\cap \langle S_1, S_2, S_3, S_4\rangle$ by $\mathcal{A}$. We fix $\Gamma$ a torsion-free finite index subgroup in $\mathcal{A}$. The limit set $\Lambda(\Gamma)$ of $\Gamma$ is exactly the closure of $\mathcal{P}_0$ (Proposition 2.9 in \cite{KO}). As a result, the critical exponent of $\Gamma$ equals the Hausdorff  dimension of the residual set of $\mathcal{P}_0$, which is denoted by $\alpha$. 

\begin{figure}[h]
\includegraphics[scale=0.3]{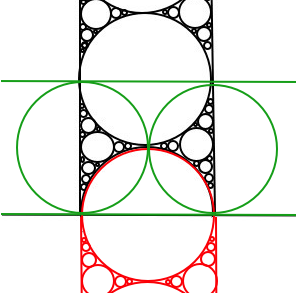}
\caption{Apollonian circle packing $\mathcal{P}$ and generators of $\mathcal{A}$}
\end{figure}

Throughout this section, we let $\eta$ be a constant in $(0,\frac{1}{4})$ satisfying the constraints described in $(\ref{constraint 1})$, $(\ref{constraint 2})$ and $(\ref{constraint 3})$.  For all small $t>0$, set
\begin{equation}
\label{restricted triangle}
\mathcal{T}({\eta,t}):=\{x+iy \in \mathcal{T}: t^{\eta} \leq y \leq t^{-\eta}\}.
\end{equation}
In fact,  an elementary computation in hyperbolic geometry shows that
\begin{equation*}
N_t(\mathcal{P}(\mathcal{T}))=\#\{C\in \mathcal{P}(\mathcal{T}): C\cap \mathcal{T}(\frac{1}{4}, t)\neq \emptyset, \operatorname{Area}_{\operatorname{hyp}}(C)>t\}.
\end{equation*}
However, it will be clear from the proof that $\mathcal{T}(\frac{1}{4},t)$ is not the right region to consider in order to obtain an effective estimate of $N_t(\mathcal{P}(\mathcal{T}))$.

\begin{notn}
For any circle packing $\mathcal{P}$ in $\hat{\mathbb{C}}$ and subset $E\subset \mathbb{H}^2$, set
 \begin{equation*}
 \mathcal{P}\cap E:=\{C\in \mathcal{P}:C\cap E\neq \emptyset\}.
 \end{equation*}
\end{notn}

\subsection{Reformulation into orbit counting problem}


Let $C$ be a hyperbolic circle in $\mathbb{H}^2$. Note that $C$ is also an Euclidean circle. Denote the Euclidean center and the Euclidean radius of $C$ by $e_C$ and $r_C$ respectively. We have the following equivalent relation through a basic computation in hyperbolic geometry:
\begin{equation}
\label{equivalent area}
\operatorname{Area}_{\operatorname{hyp}}(C)>t\Longleftrightarrow r_{C}>\operatorname{Im}(e_C)\cdot \beta(t),
\end{equation}
where $\beta(t)=\frac{\sqrt{t(4\pi+t)}}{(2\pi+t)}$.

Fix a circle packing $\mathcal{P}$ contained in $\mathcal{T}$ in the following. The counting function $N_t(\mathcal{P})$ can be reformulated as follows using (\ref{equivalent area}):
\begin{equation}
\label{volume counting formula}
N_t(\mathcal{P})=\#\{C\in \mathcal{P}:\,r_C >\operatorname{Im}(e_C)\cdot\beta(t)\}.
\end{equation}

We introduce the following functions on $\mathbb{H}^2$: for every small $\epsilon>0$,
\begin{align}
\label{height function}
h^{+}_{t}(z) &=-\log(\beta(t)\operatorname{Im}(z))\\
h^{-}_{t}(z) &=-\log((1+t)\beta(t)\operatorname{Im}(z)).\nonumber
\end{align}
For every subset $E\subset\mathcal{T}$, define
\begin{align}
\label{sets in NA}
P_{+}(E)&:=\{a_sn_{-z}:z\in E\,\,\text{and}\,\,0\leq s\leq h^{+}_{t}(z) \},\\
P_{-}(E) &:=\{a_sn_{-z}:z\in E,\,\,\text{and}\,\,0\leq s\leq h^{-}_{t}(z)\},\nonumber\\
B_{+}(E)&=\left( P_{+}(E)\right)^{-1}\,\,\text{and}\,\,B_{-}(E)=\left(P_{-}(E)\right)^{-1}.\nonumber
\end{align}

\begin{lem}
\label{upper bound for counting}
For any small $t>0$, the following holds
\begin{equation*}
N_t(\mathcal{P}\cap \mathcal{T}(\eta,t)) \leq  \# \left\{C\in \mathcal{P}: \hat{C} \cap B_{+}(\mathcal{T}(\eta,t))j \neq \emptyset\right\}+n(\mathcal{P},t),
\end{equation*}
where $n(\mathcal{P},t):=\#\{C\in \mathcal{P} :\operatorname{Area}_{\operatorname{hyp}}(C)>t,\,e_C \notin \mathcal{T}(\eta,t)
\}$.
\end{lem}
As $B_{+}(\mathcal{T}(\eta,t))j=\{z+rj\in \mathbb{H}^3: z\in \mathcal{T}(\eta,t), \,\beta(t)\operatorname{Im}(z)< r\leq 1\}$, this lemma can be easily verified using  (\ref{equivalent area}).

\begin{lem}
\label{lower bound for counting} 
For any small $t>0$, we have
\begin{equation*}
N_t(\mathcal{P}\cap \mathcal{T}(\eta,t)) \geq  \# \left\{C\in \mathcal{P}: \hat{C} \cap B_{-}(\mathcal{T}(\eta,t))j \neq \emptyset\right\}.
\end{equation*}
\end{lem}

\begin{proof}
Let $C$ be any circle in $\mathcal{P}$ such that $\hat{C} \cap B_{-}(\mathcal{T}(\eta,t))j \neq \emptyset$. Suppose $e_C=x_0+iy_0$. We may assume that  $x_0+i(y_0+r_C \cos \theta_0)+j r_C\sin \theta_0$ is the point at which $\hat{C}$ and $B_{-}(\mathcal{T}(\eta,t))j$ intersect. We claim that $C\cap \mathcal{T}(\eta,t)\neq \emptyset$ and $\operatorname{Area}_{\operatorname{hyp}}(C)>t$.  The claim $C\cap \mathcal{T}(\eta,t)\neq \emptyset$ directly follows from the observation that $B_{-}(\mathcal{T}(\eta,t))j=\{z+rj\in\mathbb{H}^3: z\in \mathcal{T}(\eta,t), (1+t)\beta(t)\operatorname{Im}(z)<r\leq 1\}$. Now we show the second claim. In view of  (\ref{equivalent area}), it suffices to show that $r_C>y_0\beta(t)$.

Consider the function $f$  given by $f(\theta)=r_C\sin \theta-(1+t)\beta(t)(y_0+r_C\cos \theta)$ for $\theta\in (0,\pi)$. Since $x_0+i(y_0+r_C\cos \theta_0)+jr_C\sin \theta_0$ is in $B_{-}(\mathcal{T}(\eta,t))j$, $f(\theta_0)>0$. There are two possibilities to discuss.

First suppose $\theta_0 \in [\frac{\pi}{2},\pi)$.  Then $\max f|_{[\frac{\pi}{2},\pi)}=r_C(1+(1+t)^2\beta(t)^2)^{\frac{1}{2}}-y_0(1+t)\beta(t)>0$. When $t$ is small enough, we have 
\begin{equation*}
({1+(1+t)^2\beta(t)^2})^{\frac{1}{2}} < 1+t.
\end{equation*}
Consequently,
\begin{equation*}
r_C > \frac{y_0(1+t)\beta(t)}{(1+(1+t)^2\beta(t)^2)^{\frac{1}{2}}}> y_0\beta(t).
\end{equation*}

Next we suppose that $\theta_0 \in (0,\frac{\pi}{2})$. Since the derivative of $f$ is always positive on $(0,\frac{\pi}{2}]$ and $f(\theta_0)>0$, we have $f(\frac{\pi}{2})>0$, i.e. $r_C>(1+t)\beta(t)y_0>\beta(t)y_0$, completing the proof of the second claim.
\end{proof}

Applying Lemmas \ref{upper bound for counting} and \ref{lower bound for counting} to  $\Gamma (C_0)\cap \mathcal{T}(\eta, t)$, we relate the circle counting function to the following orbit counting functions:
\begin{prop}
\label{reformulation inequality}
For small $t>0$, the following inequalities hold:
\begin{align*}
&\# \left([e]\Gamma \cap H\backslash HKP_{-}(\mathcal{T}(\eta,t))\right)\leq N_{t}\left(\Gamma(C_0)\cap \mathcal{T}(\eta,t)\right), \\
&N_{t}\left(\Gamma(C_0)\cap \mathcal{T}(\eta,t)\right)\leq \# \left([e]\Gamma \cap H\backslash HKP_{+}(\mathcal{T}(\eta,t))\right)+n(\Gamma(C_0),t),
\end{align*}
where $n(\Gamma(C_0),t)=\#\{C\in \Gamma(C_0):\operatorname{Area}_{\operatorname{hyp}}(C)>t, e_C\notin \mathcal{T}(\eta,t)\}$.
\end{prop}

\subsection{Number of disks in the cuspidal neighborhoods}
We estimate the counting function $n(\mathcal{P}_0,t)$ defined in Lemma \ref{upper bound for counting}:
\begin{prop}
\label{circles in the cusp 2}
For all sufficiently small $t>0$, we have
\begin{equation*}
n(\mathcal{P}_0,t)=O(t^{-\frac{\alpha}{2}+\eta(\alpha-1)}).
\end{equation*}
\end{prop}

\begin{lem}
\label{circles in the cusp 1}
For all sufficiently small $t>0$, we have
\begin{equation*}
N_t\left(\mathcal{P}_0\cap\{|z|\geq t^{-\eta}\}\right)=O({t^{-\frac{\alpha}{2}+\eta (\alpha-1)}}).
\end{equation*}
\end{lem}

\begin{proof}
Let $C$ be any circle in $\mathcal{P}_0$ such that $C$ intersects $\{| z|\geq t^{-\eta}\}$. Since $\mathcal{P}_0$ is periodic, we choose $C' \subset \mathcal{P}_0\cap\{0<\operatorname{Im}z\leq 2\}$ so that $C=C'+2k$ for some $k\in \mathbb{N}$. Denote the Euclidean radius of $C'$ by $r_0$. Then
\begin{equation*}
\operatorname{Area}_{\operatorname{hyp}}(C)= \frac{\pi r_0^2}{4k^2}\left(1+O(k^{-1})\right). \nonumber
\end{equation*}
It follows from Theorem \ref{effective circle counting thm} that
\begin{align*}
& N_{t}(\mathcal{P}_0\cap\{|z|\geq t^{-\eta}\})\\ 
\ll& \sum_{k \geq \frac{t^{-\eta}}{2}} \# \left\{C'\in \mathcal{P}_0\cap\{0<\operatorname{Im}z\leq 2\}: \operatorname{Curv}(C)<\sqrt{\frac{\pi}{4k^2 t}}\right\}\\ 
\ll& \sum_{k\geq t^{-\eta}/2} \left(\frac{\pi}{4k^2 t}\right)^{\frac{\alpha}{2}}\ll 
 t^{-\frac{\alpha}{2}+\eta (\alpha-1)}.
\end{align*}
\end{proof}

\begin{lem}
\label{circles in the cusp 3}
For all sufficiently small $t>0$, we have
\begin{align*}
N_t\left(\mathcal{P}_0\cap\{| z\pm 1|\leq t^{\eta}\}\right)&=O(t^{-\frac{\alpha}{2}+\eta(\alpha-1)}).
\end{align*}
\end{lem}

\begin{proof}
Let $g_0:=\begin{pmatrix} \frac{1}{2} & -\frac{3}{2} \\ \frac{1}{2} & \frac{1}{2}\end{pmatrix} \in \operatorname{PSL}_2(\mathbb{R})$. 
As $g_0$ preserves $\mathcal{P}_0$ as well as the hyperbolic metric, we have
\begin{equation*}
N_t(\mathcal{P}_0\cap\{| z-1|\leq t^{\eta}\})=N_t(\mathcal{P}_0\cap\{| z|\geq ct^{-\eta}\}),
\end{equation*}
where $c>0$ is some constant. As a result, the estimate of $N_t(\mathcal{P}_0\cap \{|z-1|\leq t^{\eta}\})$ easily follows from Lemma \ref{circles in the cusp 1}. The estimate of $N_t(\mathcal{P}_0\cap\{|z+1|\leq t^{\eta}\})$ can be verified similarly using $g_0^{-1}$.
\end{proof}

In view of the inclusion
\begin{equation*}
\mathcal{T}\backslash \mathcal{T}(\eta,t) \subset \{|z|\geq t^{-\eta}\} \cup \{| z-1|\leq 2t^{\eta}\}\cup \{| z+1|\leq 2t^{\eta}\},
\end{equation*}
Proposition \ref{circles in the cusp 2} directly follows from Lemmas \ref{circles in the cusp 1} and \ref{circles in the cusp 3}.

\subsection{On the measure $(\operatorname{Im}z)^{-\alpha}d\omega_{\Gamma}$} 
We utilize the measure $d\omega_{\Gamma}$ (Definition \ref{measure on the complex plane}) when we count the number of circles with respect to the Euclidean metric. As now we are counting the number of circles with respect to the hyperbolic metric,  we introduce a modification of $d\omega_{\Gamma}$,  $(\operatorname{Im}z)^{-\alpha}d\omega_{\Gamma}$ on $\mathbb{H}^2$.

\subsubsection{Estimate of the measure of cuspidal neighborhoods}

\begin{thm}[Theorem 9.3 in \cite{Harmonic analysis}]
We have
\begin{equation*}
\int_{z\in \mathcal{T}}(\operatorname{Im}z)^{-\alpha}d\omega_{\Gamma}(z)<\infty.
\end{equation*}
\end{thm}

Recall that $\nu_j$ is the PS-density at $j$.
\begin{prop}
\label{PS density at cusp}
There exists  $N_1\geq 1$ such that for all $N\geq N_1$, we have
\begin{equation*}
\nu_j(\mathcal{E}_{N})\ll N^{-2\alpha+1},
\end{equation*}
where $\mathcal{E}_{N}=\{z\in \mathbb{C}:|z|\geq N\}$.
\end{prop}

\begin{proof}
Because the limit set $\Lambda (\Gamma)$ of $\Gamma$ is exactly  the closure of $\mathcal{P}_0$,  we have
\begin{equation*}
\Lambda(\Gamma)\backslash \{\infty\}\subset \{|\operatorname{Re}z|\leq 1\}.
\end{equation*}

The stabilizer of $\infty$ in $\Gamma$ is generated by $\gamma_0=\begin{pmatrix} 1 & i y_0\\0 & 1\end{pmatrix}$ for some $y_0>0$. Define the following relatively compact set in $\mathcal{F}$:
\begin{equation*}
\mathcal{F}=\{z:|\operatorname{Re}z|\leq 1,\,\text{and}\,0\leq \operatorname{Im}z<y_0\}.
\end{equation*}
For any $z\in\Lambda (\Gamma)\backslash\{\infty\}$, there exists a unique $k\in \mathbb{Z}$ such that  $z\in \gamma_0^k\mathcal{F}$. This yields the estimate
\begin{equation*}
y_0^2\cdot(|k|+1)^2\geq |z|^2-|\operatorname{Re}z|^2\geq  |z|^2-1.
\end{equation*}
Therefore, there exist $c\geq 1$ and $N_1\geq 1$ such that for all $N\geq N_1$, we have
\begin{equation*}
\mathcal{E}_{cN}\cap \Lambda (\Gamma)\subset \bigcup_{|k|\geq N}\gamma_0^k\mathcal{F}.
\end{equation*}

We continue to prove the proposition. By the above inclusion, for any $N\geq N_1$, we have
\begin{align*}
\nu_j(\mathcal{E}_{cN}) & \leq \sum_{k\geq N}\nu_j(\gamma_0^k \mathcal{F})+\sum_{k\geq N}\nu_j(\gamma_0^{-k}\mathcal{F})\\
&= \sum_{k\geq N} \nu_{\gamma_0^{-k}j}(\mathcal{F})+\sum_{k\geq N}\nu_{\gamma_0^kj}(\mathcal{F})\\
&= \sum_{k\geq N} \int_{z\in \mathcal{F}} e^{-\alpha \beta_z(\gamma_0^{-k}j,j)}d\nu_j(z)+\sum_{k\geq N} \int_{z\in \mathcal{F}}e^{-\alpha \beta_z(\gamma_0^{k}j,j)}d\nu_j(z)\\
&\ll \sum_{k\geq N} k^{-2\alpha}\nu_{j}(\mathcal{F})+\sum_{k\geq N} k^{-2\alpha}\nu_j(\mathcal{F}) \;\;\; (\text{by Lemma}\,\,\ref{Busemann function})\\
&\ll N^{-2\alpha+1}.
\end{align*}
\end{proof}

\begin{prop}
\label{hyperbolic PS density at cusp}
There exists $T_0>1$ such that for any pair of real numbers $a,\,b$ with $T_0<a<b<\infty$ and any $\kappa$ with $\kappa>-2$ and $\kappa\neq -1$, we have
\begin{equation*}
\int_{\{z\in \mathcal{T}:a<\operatorname{Im}z<b\}} (\operatorname{Im}z)^{\kappa}d\omega_{\Gamma}(z)=O\left(\frac{a^{\kappa+1}+b^{\kappa+1}}{|\kappa+1|}\right).
\end{equation*}
\end{prop}

\begin{proof}
For $a>0$  large enough, using Lemma \ref{Busemann function}, we have
\begin{align*}
\int_{ \{z\in \mathcal{T}:a<\operatorname{Im}z<b\}} (\operatorname{Im}z)^{\kappa}d\omega_{\Gamma}(z) &=\int_{ \{z\in\mathcal{T}:a<\operatorname{Im}z<b\}}(\operatorname{Im}z)^{\kappa}(|z|^2+1)^{\alpha}d\nu_j(z)\\
&\ll \int_{\{z\in \mathcal{T}:a<\operatorname{Im}z<b\}} (\operatorname{Im}z)^{\kappa+2\alpha}d\nu_j(z).
\end{align*}

Denote by $\pi$ the projection map from points in $\{z\in \mathcal{T}:a<\operatorname{Im}z<b\}$ to their $y$-coordinates. Let $\pi_{*}\nu_j$ be the pushforward of $\nu_j$. Letting $\mathcal{E}_s=\{z\in \mathbb{C}:|z|>s\}$ for $s>0$, we have 
\begin{align*}
&\int_{\{z\in\mathcal{T}:a<\operatorname{Im}z<b\}}(\operatorname{Im}z)^{\kappa+2\alpha}d\nu_j(z)\\
=&\int_{a}^{b} y^{\kappa+2\alpha} d\pi_*\nu_j(y)\\
\ll & \int_{a}^{b}\int_{0}^{y}s^{\kappa+2\alpha-1} ds d\pi_{*}\nu_j(y)\\
= & \int_{0}^{a}\int_{a}^{b}s^{\kappa+2\alpha-1}d\pi_*\nu_j(y)ds+\int_{a}^{b}\int_{s}^{b}s^{\kappa+2\alpha-1}d\pi_*\nu_j(y)ds \\
\leq  &\int_{0}^{a}s^{\kappa+2\alpha-1}\nu_{j}(\mathcal{E}_{a})ds+\int_{a}^{b}s^{\kappa+2\alpha-1}\nu_{j}(\mathcal{E}_{s})ds \\\ll & \int_{0}^{a}s^{\kappa+2\alpha-1}\cdot a^{1-2\alpha}ds+\int_{a}^{b} s^{\kappa+2\alpha-1}\cdot s^{1-2\alpha}ds  \;\;(\text{by Proposition}\, \ref{PS density at cusp})\\
\ll & |\kappa+1|^{-1}\cdot(a^{\kappa+1}+b^{\kappa+1}).
\end{align*}
\end{proof}

Now we estimate the size of the cuspidal neighborhoods under the measure $(\operatorname{Im}z)^{-\alpha}d\omega_{\Gamma}$.
\begin{cor}
\label{measure of the cusp}
For all small $t>0$, we have
\begin{equation*}
\int_{z\in \{w\in\mathcal{T}:\operatorname{Im}w>t^{-\eta}\}}(\operatorname{Im}z)^{-\alpha}d\omega_{\Gamma}(z)=O(t^{\eta(\alpha-1)}),
\end{equation*}
where $\alpha$ is the critical exponent of $\Gamma$ and $\eta$ is defined as (\ref{restricted triangle}).
\end{cor}
The corollary can be proved by setting $k=-\alpha$ in Proposition \ref{hyperbolic PS density at cusp} and taking the limit of $\int_{\{z\in\mathcal{T}:t^{-\eta}<\operatorname{Im}z<N\}}(\operatorname{Im}z)^{-\alpha}d\omega_{\Gamma}(z)$ as $N\to \infty$.

\begin{cor}
\label{measure at 1,-1}
For all sufficiently small $t>0$, we have
\begin{align*} 
\int_{\{z\in\mathcal{T}:|z\pm 1|\leq t^{\eta}\}} (\operatorname{Im}z)^{-\alpha}d\omega_{\Gamma}(z)&=O(t^{\eta (\alpha-1)}).
\end{align*}
\end{cor}

\begin{proof}
As observed in Lemma \ref{circles in the cusp 3}, $g_0^{-1}$ maps $\{z\in \mathcal{T}: |z-1|\leq t^{\eta}\}$ to $\{z\in \mathcal{T}:\operatorname{Im}z\geq ct^{-\eta}\}$ for some constant $c>0$. 

For every $x\in \mathbb{H}^3$, set
\begin{equation*}
\tilde{\nu}_x:=(g_0)^{*}\nu_{\Gamma,g_0x},
\end{equation*}
where $g_0^{*}\nu_{\Gamma, g_0(x)}(R)=\nu_{\Gamma, g_0(x)}(g_0(R))$. It is easy to check that $\{\tilde{\nu}_x:x\in \mathbb{H}^3\}$ is a $g_0^{-1} \Gamma g_0$-invariant conformal density of dimension $\delta_{g_0^{-1}\Gamma g_0}=\delta_{\Gamma}=\alpha$.

Using these two observations, we obtain
\begin{align*}
&\int_{\{z\in\mathcal{T}:|z-1|\leq t^{\eta}\}} (\operatorname{Im}z)^{-\alpha}d\omega_{\Gamma}(z)\\
=&\int_{\{z\in \mathcal{T}:\operatorname{Im}z\geq ct^{-\eta}\}} (\operatorname{Im}(g_0 z))^{-\alpha}e^{\alpha \beta_{g_0z}(j,g_0z+j)}d\nu_{g_0^{-1}\Gamma g_0, g_0^{-1}j}(z).
\end{align*}
We can get an estimate of $\nu_{g_0^{-1}\Gamma g_0,g_0^{-1}j}(\{z\in\mathcal{T}: \operatorname{Im}z\geq ct^{-\eta}\})$ using the proof of Proposition \ref{PS density at cusp}. The corollary can be shown via a similar argument as the proof of Proposition \ref{hyperbolic PS density at cusp}.
 

\end{proof}

Noting that $d\omega_{\Gamma}$ is a locally finite measure, we obtain the following corollary by setting $\kappa=0$ in Proposition \ref{hyperbolic PS density at cusp} : 
\begin{cor}
\label{measure of finite region}
For all small $t>0$, we have
\begin{equation*}
\omega_{\Gamma}(\mathcal{T}(\eta,t))=O(t^{-\eta}).
\end{equation*}
\end{cor}

\subsubsection{Relate $m^{\operatorname{BR}}$ with $(\operatorname{Im}z)^{-\alpha}d\omega_{\Gamma}$}

For small $\epsilon>0$, let $U_{\epsilon}$ be the symmetric $\epsilon$-neighborhood of $e$ in G. Recall that we defined
\begin{equation*}
\mathcal{T}(\eta,t)^{+}_{\epsilon}=\bigcup_{u\in U_{\epsilon}}u\mathcal{T}(\eta,t) \quad \text{and}\quad \mathcal{T}(\eta,t)^{-}_{\epsilon}=\bigcap_{u\in U_{\epsilon}}u\mathcal{T}(\eta,t).
\end{equation*}
Let $\psi^{\epsilon}$ be a non-negative function in $C^{\infty}_{c}(G)$ supported in $U_{\epsilon}$ with integral one. Set $\Psi^{\epsilon}\in C^{\infty}_{c}(\Gamma \backslash G)$ to be the $\Gamma$-average of $\psi^{\epsilon}$. 

We establish the following relation between $m^{\operatorname{BR}}$ and $(\operatorname{Im}z)^{-\alpha}d\omega_{\Gamma}$:
\begin{prop}
\label{translation between measures}
For all small $\epsilon>0$ and $t>0$, we have
\begin{equation*}
\int_{z\in \mathcal{T}(\eta,t)} m^{\operatorname{BR}}(\Psi^{\epsilon}_{-z})(\operatorname{Im}z)^{-\alpha}dn_{-z}=(1+O(\epsilon \cdot t^{-\eta}))\int_{z\in \mathcal{T}(\eta,t)^{\pm}_{\epsilon}}(\operatorname{Im}z)^{-\alpha}d\omega_{\Gamma}(z),
\end{equation*}
where $\Psi^{\epsilon}_{-z} \in C^{\infty}_c(\Gamma \backslash G)^{M}$ is given by $\Psi^{\epsilon}_{-z}(g):=\int_{m\in M} \Psi^{\epsilon}(gmn_{-z})dm$.
\end{prop}

We first introduce a function on $G$:
\begin{defn} For $\psi \in C_c(\mathbb{H}^2)$, define a function $\mathcal{F}_{\psi}$ on $MAN^{-}N \subset G$ by
\begin{equation*}
\mathcal{F}_{\psi}(ma_tn^{-}_{x}n_{z})=\begin{cases} e^{-\alpha t}\psi(-z)(-y)^{-\alpha},\,\mbox{if $z=x+iy$ with $y<0$}\\
0,\, \mbox{otherwise}.
\end{cases}
\end{equation*}
\end{defn}
Since the product map $M \times A \times N^{-} \times N \to G$ has a diffeomorphic image,  $\mathcal{F}_{\psi}$ can be regarded as a function on $G$.

\begin{lem}
For any $\psi \in C_c(\mathbb{H}^2)$,
\begin{equation*}
\int_{z\in \mathbb{H}^2}\psi(z)(\operatorname{Im}z)^{-\alpha}d\omega_{\Gamma}(z)=\int_{k\in K/M} \mathcal{F}_{\psi}(k^{-1}) d \nu_j(k(0)).
\end{equation*}
\end{lem}
This lemma follows from a verbatim repetition of the proof of Proposition 5.4 in \cite{Asymptotic}.

\begin{lem}
\label{lemma for measure}
For small $\epsilon>0$ and any $g\in U_{\epsilon}$,
\begin{equation*}
\int_{k\in K/M} \mathcal{F}_{\mathcal{T}(\eta,t)}(k^{-1}g)d\nu_j(k(0))=(1+O(\epsilon\cdot t^{-\eta}))\int_{z\in \mathcal{T}(\eta,t)^{\pm}_{\epsilon}}(\operatorname{Im}z)^{-\alpha}d\omega_{\Gamma}(z),
\end{equation*}
where $\mathcal{F}_{\mathcal{T}(\eta,t)}:=\mathcal{F}_{\chi_{\mathcal{T}(\eta,t)}}$ with $\chi_{\mathcal{T}(\eta,t)}$ the characteristic function of $\mathcal{T}(\eta,t)$.
\end{lem}

\begin{proof}
Write $k^{-1}=m_{\theta}a_tn^{-}_{w}n_{z}$ and $g=m_{\theta_1}a_{t_1}n^{-}_{w_1}n_{z_1} \in U_{\epsilon}$. By Lemma \ref{coordinates for product}, we have $k^{-1}g=m_{\theta_0}a_{t_0}n^{-}_{w_0}n_{z_0}$, where $t_0=t+t_1+2\log(|1+e^{-t_1-2i\theta_1}w_1z|)$ and $z_0=e^{-t_1-2i\theta_1}z(1+w_1e^{-t_1-2i\theta_1}z)^{-1}+z_1$. 
We have
\begin{align*}
&\int_{k\in K/M} \mathcal{F}_{\mathcal{T}(\eta,t)}(k^{-1}g)d\nu_j(k(0))
\\ =& \int _{k\in K/M} e^{-\alpha t_0} \chi_{\mathcal{T}(\eta,t)}(g^{-1}k(0))\operatorname{Im}(-z_0)^{-\alpha}d\nu_j(k(0))\\
=&\int_{k(0)\in \mathcal{T}(\eta,t)^{\pm}_{\epsilon}}e^{-\alpha t_0}\operatorname{Im}(-z_0)^{-\alpha} d\nu_j(k(0)).
\end{align*}
Using the estimate
\begin{equation*}
e^{-\alpha t_0}=(1+O(|z|\epsilon))e^{-\alpha t}\,\,\text{and}\,\,z_0=(1+O((|z|+|z|^{-1})\epsilon))z,
\end{equation*}
we reach the conclusion that
\begin{align*}
\int_{k\in K/M} \mathcal{F}_{\mathcal{T}(\eta,t)}(k^{-1}g) d\nu_j(k(0))
=(1+O(\epsilon\, t^{-\eta}))\int_{z\in\mathcal{T}(\eta,t)^{\pm}_{\epsilon}}(\operatorname{Im}z)^{-\alpha}d\omega_{\Gamma}(z). 
\end{align*}
\end{proof}

With Lemma \ref{lemma for measure} available, Proposition \ref{translation between measures} follows from the same argument as the proof of Lemma 5.7 in \cite{Asymptotic}.

\subsection{Conclusion}

\begin{thm}
\label{effective volume counting}
There exists a constant $\rho>0$, such that  as $t \to 0$, we have
\begin{equation*}
N_t(\mathcal{P}(\mathcal{T}))=\frac{\operatorname{sk}_{\Gamma}(\mathcal{P}_0)}{\alpha|m^{\operatorname{BMS}}_{\Gamma}|}\,\left(\frac{t}{\pi}\right)^{-\frac{\alpha}{2}}\,\int_{\overline{\mathcal{P}(\mathcal{T})}}(\operatorname{Im}z)^{-\alpha}d\mathcal{H}^{\alpha}(z)+O(t^{-\frac{\alpha}{2}+\rho}),
\end{equation*}
where $d\mathcal{H}^{\alpha}$ is the $\alpha$-dimensional Hausdorff measure on $\overline{\mathcal{P}(\mathcal{T})}$.
\end{thm}

 Set $\xi=t^{4\eta}$ for the rest of the paper. For any subset $E\subset \mathcal{T}$ and every small $t>0$, define
 \begin{align*}
 Q_t^{+}(E) &:=\{a_sn_{-z}:z\in E\,\,\text{and}\,\,-t^{\eta}\leq s\leq h^{+}_{t}(z)+t^{\eta}\},\\
 Q_t^{-}(E) &:=\{a_sn_{-z}:z\in E\,\,\text{and}\,\,t^{\eta}\leq s\leq h^{-}_{t}(z)-t^{\eta}\},
 \end{align*}
 where $h^{+}_{t}(\cdot)$ and $h^{-}_{t}(\cdot)$ are the functions defined as (\ref{height function}). We first show the following lemma:

\begin{lem}
\label{enlarge the set}
There exists $c_4>0$, such that for all sufficiently small $t>0$, we have
\begin{align*}
& KP_{+}(\mathcal{T}(\eta,t))U_{c_4\xi} \subset  KQ^{+}_{t}(\mathcal{T}(\eta,t)^{+}_{t^{2\eta}}),\\
& KQ^{-}_{t}(\mathcal{T}(\eta,t)^{-}_{t^{2\eta}})U_{c_4\xi} \subset KP_{-}(\mathcal{T}(\eta,t)),
\end{align*}
where $P_{+}(\cdot)$ and $P_{-}(\cdot)$ are  defined as (\ref{sets in NA}).
\end{lem}

\begin{proof}
Up to a uniform Lipschitz constant, we may write $U_{\xi}=N^{-}_{\xi}M_{\xi}A_{\xi}N_{\xi}$. It suffices to show 
\begin{align*}
KP_{+}(\mathcal{T}(\eta,t))M_{\xi}A_{\xi}N_{\xi}^{-} &\subset KQ^{+}_{t}(\mathcal{T}(\eta,t)^{+}_{t^{2\eta}}).
\end{align*}

Fix $z\in \mathcal{T}(\eta,t)$.  For any $m_{\theta}a_s\in M_{\xi}A_{\xi}$,
\begin{equation*}
n_{-z}m_{\theta}a_{s}=m_{\theta}a_{s}n_{-(z+z')},
\end{equation*}
where $|z'|=O(t^{3\eta})$. Note that
\begin{equation*}
h^{+}_{t}(z+z')=h^{+}_{t}(z)+O(t^{2\eta}).
\end{equation*}
Hence 
\begin{align*}
P_{+}(\mathcal{T}(\eta,t))M_{\xi}A_{\xi} &\subset KQ^{+}_{t}(\mathcal{T}(\eta,t)^{+}_{t^{2\eta}}).
\end{align*}

For $n_{w}^{-}\in N^{-}_{\xi}$ and $a_s$ with $0\leq s\leq h^{+}_{t}(z)$,
\begin{equation*}
a_sn_{-z}n^{-}_{w}=a_sn^{-}_{\frac{w}{1-zw}}\begin{pmatrix} 1-wz & 0\\ 0 & (1-wz)^{-1}\end{pmatrix} n_{\frac{-z}{1-zw}}.
\end{equation*}
Since $n^{-}_{\frac{w}{1-zw}}\in U_{\xi}$, we have $a_sn^{-}_{\frac{w}{1-zw}}a_{-s}\in U_{\xi}$. Up to a uniform Lipschitz constant, we may write $a_sn^{-}_{\frac{w}{1-zw}}a_{-s}=ka_{s_1}n_{z_1}$ with $k\in K_{\xi}$, $a_{s_1}\in A_{\xi}$ and $n_{z_1}\in N_{\xi}$. Then

\begin{equation*}
a_sn_{-z}n^{-}_{w}=ka_{s+s_1}\begin{pmatrix} 1-wz & 0\\ 0 & (1-wz)^{-1}\end{pmatrix}n_{-(z+z_2)},
\end{equation*}
with $|z_2|=O(t^{2\eta})$. Therefore
\begin{align*}
P_{+}(\mathcal{T}(\eta,t))N_{\xi}^{-} &\subset KQ^{+}_{t}(\mathcal{T}(\eta,t)^{+}_{t^{2\eta}}).
\end{align*}

The second statement can be proved similarly.
\end{proof}

\begin{proof}[\textbf{Proof of Theorem 4.22}]
\textbf{Step 1:} First we consider $N_t(\Gamma(C_0)\cap\mathcal{T})$.

For simplicity, we set
\begin{align*}
V_{+} &:=  H\backslash HKQ^{+}_{t}(\mathcal{T}(\eta,t)^{+}_{t^{2\eta}}),\\
V_{-} &:= H\backslash HKQ^{-}_{t}(\mathcal{T}(\eta,t)^{-}_{t^{2\eta}}).
\end{align*}

For small $t>0$, define functions $F^{\pm}_{t}$ on $\Gamma \backslash G$:
\begin{align*}
F^{+}_{t}(g)&:=\sum_{\gamma \in \Gamma_{H}\backslash \Gamma}\chi_{V_{+}}([e]\gamma g),\\
F^{-}_{t}(g) &:=\sum_{\gamma \in \Gamma_{H}\backslash \Gamma} \chi_{V_{-}}([e]\gamma g).
\end{align*}
We deduce from  Propositions \ref{reformulation inequality}, \ref{circles in the cusp 2} and Lemma \ref{enlarge the set} that for all small $t>0$ and $g\in U_{c_4\xi}$ with $\xi=t^{4\eta}$ 
\begin{equation}
\label{sandwich counting formula}
F^{-}_{t}(g)\leq N_{t}(\Gamma(C_0)\cap \mathcal{T}(\eta,t)) \leq F^{+}_{t}(g)+\rho_1\cdot t^{-\frac{\alpha}{2}+\eta(\alpha-1)},
\end{equation}
where $\rho_1>0$ is some constant (cf. Lemma 6.4 in \cite{Asymptotic}).

Let $\psi^{\xi}$ be a non-negative function in $C_c^{\infty}(G)$, supported in $U_{c_4\xi}$ with integral one. Set $\Psi^{\xi}\in C_{c}^{\infty}(\Gamma \backslash G)$ to be the $\Gamma$-average of $\psi^{\xi}$. Integrating $F^{+}_{t}$ against $\Psi^{\xi}$, we get
\begin{align*}
\langle F^{+}_{t}, \Psi^{\xi} \rangle 
&=\int_{\Gamma\backslash G} \sum_{\gamma \in \Gamma_{H} \backslash \Gamma} \chi_{V_{+}}([e]\gamma g)\Psi^{\xi}(g)dg\\
&=\int_{g\in \Gamma_{H} \backslash G}\chi_{V_{+}}([e]g)\Psi^{\xi}(g) dg\\
&=\int_{g\in V_{+}}\int_{h\in \Gamma_{H} \backslash C^{\dagger}_0}\int_{m\in M}\Psi^{\xi}(hmg)dmdh d\lambda(g).
\end{align*}

Recall the function $h^{+}_{t}(\cdot)$ defined in (\ref{height function}).  The first constraint for $\eta$ is that
\begin{equation}
\label{constraint 1} 
0<\eta<1/10. 
\end{equation}
This constraint guarantees that for every $z\in \mathcal{T}(\eta,t)^{+}_{t^{2\eta}}$, 
\begin{equation*}
 -\log(c_1\xi)< h^{+}_{t}(z),
 \end{equation*}
 where $c_1$ is the constant described at the beginning of Section 3.2. 
 
 Set
 \begin{align}
 \label{decomposition}
 & T_0:=-\log(c_1\xi), \,\, \hat{A}(z):=\{a_s:T_0-2t^{\eta}\leq s\leq h^{+}_{t}(z)+2t^{\eta}\},\\
 &K(s):= \{k\in K:a_sk\in HKA^{+}\}\,\,\text{for every $s>0$},\nonumber\\
 & W_1:=\bigcup_{z\in\mathcal{T}(\eta,t)^{+}_{t^{2\eta}}} H\backslash H\hat{A}(z)n_{-z},\nonumber\\
 & W_2:=\bigcup_{0\leq s\leq T_0}H\backslash Ha_{s}K(s)N_{-\mathcal{T}(\eta,t)^{+}_{t^{2\eta}}}.\nonumber
 \end{align}
  For every $z\in \mathcal{T}(\eta,t)^{+}_{t^{2\eta}}$, using Lemma \ref{structure analysis 1} (1), we get
 \begin{equation*}
 \bigcup_{-t^{\eta}\leq s\leq h^{+}_{t}(z)+t^{\eta}}H\backslash HKa_sn_{-z}=\bigcup_{0\leq s\leq h^{+}_{t}(z)+t^{\eta}}H\backslash Ha_sK(s)n_{-z}.
 \end{equation*}
 Applying the same argument as the proof of Proposition \ref{rewrite B_T(E)} to the subset $\underset{0\leq s\leq h^{+}_{t}(z)+t^{\eta}}{\bigcup}H\backslash H a_sK(s)n_{-z}$ for every $z\in \mathcal{T}(\eta,t)^{+}_{t^{2\eta}}$, we have
 \begin{equation}
V_{+}\subset W_1N_{\rho_2\xi}\bigcup W_2,
 \end{equation}
for some constant $\rho_2>0$.

For $W_1N_{\rho_2\xi}$, we have
\begin{align}
\label{integral 4}
&\int_{W_1N_{\rho_2\xi}}\int_{\Gamma_{H}\backslash C^{\dagger}_{0}} \int_{m\in M} \Psi^{\xi}(hmg)dmdh d\lambda(g)\\
\leq& \int_{z\in \mathcal{T}^{+}_{\xi}} \int^{h^{+}_{t}(z)+3t^{\eta}}_{T_0-3t^{\eta}}e^{2s}\int_{\Gamma_{H} \backslash C^{\dagger}_{0}}\int_{m\in M} \Psi^{\xi}(hma_sn_{-z}) dm dh ds dn_{-z} \nonumber \\
=&\int_{z\in \mathcal{T}^{+}_{\xi}} \int^{h^{+}_{t}(z)+3t^{\eta}}_{T_0-3t^{\eta}}e^{2s}\int_{ \Gamma_{H} \backslash C^{\dagger}_{0}}\Psi^{\xi}_{-z}(ha_s)dh ds dn_{-z}, \nonumber
\end{align}
where $\mathcal{T}^{+}_{\xi}=U_{\rho_2\xi}\mathcal{T}(\eta,t)^{+}_{t^{2\eta}}$ and $\Psi^{\xi}_{-z}(g):=\int_{m\in M} \Psi^{\xi}(gmn_{-z})dm$ for every $g\in \Gamma \backslash G/M$. 

For any $z\in \mathcal{T}^{+}_{\xi}$, we have $\operatorname{supp}(\Psi^{\xi}_{-z})\subset \Gamma \backslash \Gamma n_z U_1M/M$, with $U_1$ the 1-neighborhood of identity in $G$. Let $E:=\{x+iy \in \mathbb{C}: -2\leq x \leq 2, \, 0\leq y\leq 4\}$. We can see  $\mathcal{A}\backslash \mathcal{A}n_zU_{1}K/K\subset \mathcal{A}\backslash \mathcal{A} N_{E}U_1K/K$ by applying $S_4S_3\in\mathcal{A}$ to $N_EU_1j$ finitely many times. Hence there exists a compact subset $\Omega \subset \Gamma \backslash G/M$ such that $\operatorname{supp}(\Psi^{\xi}_{-z})\subset \Omega$. Applying Theorem \ref{effective distribution} to $\Psi^{\xi}_{-z}$, we get 
\begin{align} 
\label{integral}
&\int_{W_1N_{\rho_2\xi}}\int_{\Gamma_{H}\backslash C^{\dagger}_{0}} \int_{m\in M} \Psi^{\xi}(hmg)dmdh d\lambda(g) \\ 
&\leq \int_{z\in \mathcal{T}^{+}_{\xi}}\int ^{h^{+}_{t}(z)+3t^{\eta}}_{T_0-3t^{\eta}}\frac{|\mu^{\operatorname{PS}}_{H}|}{|m^{\operatorname{BMS}}|}m^{\operatorname{BR}}(\Psi^{\xi}_{-z})e^{\alpha s}+\rho_3 (S_l(\Psi^{\xi}_{-z})e^{(\alpha-\eta_0)s})dsdn_{-z} \notag  \\
&\leq \int_{z\in \mathcal{T}^{+}_{\xi}}  c \left(\tfrac{t}{\pi}\right)^{-\frac{\alpha}{2}} m^{\operatorname{BR}}(\Psi^{\xi}_{-z}) (\operatorname{Im}z)^{-\alpha}(1+\rho_4 t^{\eta})+\rho_4(| z|^{2l+\eta_0-\alpha}S_l(\Psi^{\xi})t^{\frac{\eta_0-\alpha}{2}})dn_{-z} \nonumber\\
&\leq  c\,(1+\rho_5 t^{\eta})\,\left(\tfrac{t}{\pi}\right)^{-\frac{\alpha}{2}}\int_{z\in \mathcal{T}^{+}_{\xi}}m^{\operatorname{BR}}(\Psi^{\xi}_{-z}) (\operatorname{Im}z)^{-\alpha}dn_{-z} \nonumber \\
&+ \rho_5 \left(t^{-\frac{\alpha}{2}+\frac{\eta_0}{2}-\eta(6l+13+\eta_0-\alpha)}\right), \nonumber
\end{align} 
where $c:=\frac{|\mu^{\operatorname{PS}}_{H}|}{\alpha \cdot |m^{\operatorname{BMS}}|}$, $\eta_0$ is the constant in Theorem \ref{effective distribution},  and $\rho_3,\rho_4,\rho_5>0$  are some constants. In fact, to get the second inequality above, we use the fact that there exists $C>0$ such that  $S_{l}(\Psi^{\xi}_{-z})\leq C\cdot (y^{2l}S_{l}(\Psi^{\xi}))$ for $z=x+iy \in \mathcal{T}$. 

The second constraint for $\eta$ is that
\begin{equation}
\label{constraint 2}
 \frac{\eta_0}{2}-\eta(6l+13+\eta_0-\alpha)>\eta. 
 \end{equation}
 Applying Proposition \ref{translation between measures} to the first term in the third inequality of (\ref{integral}), we obtain
\begin{align}
\label{integral 13}
 &c\,(1+\rho_5\cdot t^{\eta})\,\left(\frac{t}{\pi}\right)^{-\frac{\alpha}{2}}\int_{z\in \mathcal{T}^{+}_{\epsilon}}m^{\operatorname{BR}}(\Psi^{\xi}_{-z}) (\operatorname{Im}z)^{-\alpha}dn_{-z} +\rho_5\cdot (t^{-\frac{\alpha}{2}+\eta})\\
 &\leq c\, \left(\frac{t}{\pi}\right)^{-\frac{\alpha}{2}} (1+\rho_6\cdot(t^{-\eta}\xi))(1+\rho_5\cdot t^{\eta})\int_{z\in\mathcal{T}}(\operatorname{Im}z)^{-\alpha}d\omega_{\Gamma}(z)+\rho_5\cdot (t^{-\frac{\alpha}{2}+\eta}) \nonumber \\
& \leq c\,\left(\frac{t}{\pi}\right)^{-\frac{\alpha}{2}}\int_{z\in\mathcal{T}}(\operatorname{Im}z)^{-\alpha}d\omega_{\Gamma}(z) +\rho_7\cdot \left(t^{-\frac{\alpha}{2}+\eta}\right),\nonumber
\end{align}
for some constants  $\rho_6,\rho_7>0$.

For the set $W_2$ in (\ref{decomposition}), note that 
\begin{align}
\label{decomposition 2}
W_2 &= H\backslash HKA^{+}_{T_0}N_{-\mathcal{T}(\eta,t)^{+}_{t^{2\eta}}} && (\text{by Lemma}\,\, \ref{structure analysis 1})\\
&\subset  H\backslash HKA^{+}_{T_1}N_{-\mathcal{T}(\eta,t)^{+}_{t^{2\eta}}} \cup \bigcup_{T_1-\rho_8\epsilon\leq s \leq T_0+\rho_8\epsilon}H\backslash Ha_{s}N_{-\mathcal{T}_{\epsilon}} && (\text{by Prop. }\,\ref{rewrite B_T(E)})  \nonumber\\
&:=W_3\cup W_4\nonumber,
\end{align}
where $T_1=-\log(c_1\epsilon)$ for some fixed small $\epsilon>0$ such that $T_1<T_0=-\log(c_1\xi)$, $\rho_8>0$ is a constant and  $\mathcal{T}_{\epsilon}:=U_{\rho_8\epsilon}\mathcal{T}(\eta,t)^{+}_{t^{2\eta}}$.

For the set $W_4$, by similar calculation as (\ref{integral 8}), we have
\begin{align*}
&\int_{W_4}\int_{\Gamma_{H}\backslash C_0^{\dagger}} \int_{M}\Psi^{\xi}(hmg)dm dh d\lambda(g)\\
\leq & 2\,c\,e^{\alpha T_0}\int_{z\in U_{\rho_{9}\epsilon}\mathcal{T}_{\epsilon}} m^{\operatorname{BR}}(\Psi^{\xi}_{-z}) dn_{-z} && (\rho_9>0\,\,\text{some constant}).
\end{align*}

Using Proposition \ref{relation of measures} and Corollary \ref{measure of finite region} and substituting $T_0$ by $-\log (c_1\xi)$, we obtain
\begin{align}
\label{integral 6}
\int_{W_4}\int_{\Gamma_{H}\backslash C_0^{\dagger}} \int_{M}\Psi^{\xi}(hmg)dm dh d\lambda(g)=O\left(t^{-\frac{\alpha}{2}+(\frac{\alpha}{2}-(4\alpha+1)\eta)}\right).
\end{align}
The third constraint for $\eta$ is that 
\begin{equation}
\label{constraint 3}
\frac{\alpha}{2}-(4\alpha+1)\eta>\eta.
\end{equation}
 This yields
\begin{equation}
\label{integral 12}
\int_{W_4}\int_{\Gamma_{H}\backslash C_0^{\dagger}} \int_{M}\Psi^{\xi}(hmg)dm dh d\lambda(g)=O(t^{-\frac{\alpha}{2}+\eta}).
\end{equation}

For the set $W_3$ in (\ref{decomposition 2}), reserving the process of translating the circle counting with respect to Euclidean metric into orbit counting, we have
\begin{align}
\label{integral 7}
&\int_{W_3}\int_{\Gamma_{H}\backslash C_0^{\dagger}} \int_{M}\Psi^{\xi}(hmg)dm dh d\lambda(g)\\
=& O(t^{-\eta}\cdot \#\{C\in \mathcal{P}(\mathcal{T}): C\cap \mathcal{T}_{0,2}\neq \emptyset, \operatorname{Curv}(C)\leq T_1\}) \nonumber \\
=& O\left(t^{-\frac{\alpha}{2}+(\frac{\alpha}{2}-\eta)}\right),\nonumber
\end{align}
since $T_1$ is fixed.

Summing (\ref{integral 13}), (\ref{integral 12}) and (\ref{integral 7}) together,  we get 
\begin{align*}
&\langle F^{+}_{t}, \Psi^{\xi}\rangle
\leq  c \, \left(\frac{t}{\pi}\right)^{-\frac{\alpha}{2}} \int_{z\in\mathcal{T}}(\operatorname{Im}z)^{-\alpha}d\omega_{\Gamma}(z)+\rho_9\cdot(t^{-\frac{\alpha}{2}+\eta}),
\end{align*}
for some constant $\rho_9>0$. In view of (\ref{sandwich counting formula}), we conclude that there exists $\rho'>0$, such that 
\begin{align*}
N_t(\Gamma (C_0)\cap\mathcal{T}) &\leq N_t(\Gamma(C_0)\cap\mathcal{T}(\eta,t))+n(\Gamma(C_0),t)\\ 
&\leq \langle F^{+}_{t}, \Psi^{\xi} \rangle+\rho_1\cdot (t^{-\frac{\alpha}{2}+\eta(\alpha-1)})+n(\Gamma(C_0),t)\\
& \leq  c\,\left(\frac{t}{\pi}\right)^{-\frac{\alpha}{2}}\int_{z\in\mathcal{T}}(\operatorname{Im}z)^{-\alpha}d\omega_{\Gamma}(z)+\rho_{10}\cdot(t^{-\frac{\alpha}{2}+\rho'}),
\end{align*}
for some constant $\rho_{10}>0$.

As for $\langle F^{-}_{t}, \Psi^{\xi} \rangle$ in (\ref{sandwich counting formula}), observe that
\begin{equation*}
V_{-}\supset \left\{ H\backslash Ha_sn_{-z}:z\in\mathcal{T}(\eta,t)^{-}_{t^{2\eta}},\,\,T_2\leq s\leq h^{-}_{t}(z)-t^{\eta}\right\},
\end{equation*}
where $T_2>0$ is some fixed large number and $h^{-}_{t}(\cdot)$ is defined as (\ref{height function}). Using similar argument as above, there exists $\rho''>0$ such that 
\begin{align*}
\langle F^{-}_{t}, \Psi^{\xi} \rangle 
\geq &c\, \left(\frac{t}{\pi}\right)^{-\frac{\alpha}{2}}\int_{z\in\mathcal{T}}(\operatorname{Im}z)^{-\alpha}d\omega_{\Gamma}(z)
+\rho_{11} \cdot(t^{-\frac{\alpha}{2}+\rho''}),
\end{align*}
where  $\rho_{11}>0$ is some constant.

Therefore, setting $\rho=\min\{\rho',\rho''\}$, we conclude that
\begin{equation*}
N_t(\Gamma(C_0))=c\,\left(\frac{t}{\pi}\right)^{-\frac{\alpha}{2}}\,\int_{z\in\mathcal{T}}(\operatorname{Im}z)^{-\alpha}d\omega_{\Gamma}(z)+O(t^{-\frac{\alpha}{2}+\rho}).
\end{equation*}

\textbf{Step 2:} We consider other $\Gamma$-orbits in $\mathcal{P}_0$. Let $C_1$ be a representative of a $\Gamma$-orbit with Euclidean center $p=p_1+ip_2\in \mathbb{C}$ and Euclidean radius $r>0$. Set 
\begin{equation*}
g_0:=\begin{pmatrix} 1& p\\0 & 1\end{pmatrix} \begin{pmatrix} \sqrt{r} & 0\\ 0 & \frac{1}{\sqrt{r}}\end{pmatrix}.
\end{equation*}
Then $g_0^{-1}(z)=\frac{z-p}{r}$ for $z\in \mathbb{C}$ and $g_0^{-1}(C_1)=C_0$.

Setting $\Gamma_0=g_0^{-1}\Gamma g_0$, we have
\begin{align*}
N_t(\Gamma (C_1)\cap\mathcal{T}) 
&=\# \{C\in \Gamma_0 (C_0): C\cap g_0^{-1}\mathcal{T}\neq \emptyset, \operatorname{Area}_{\operatorname{hyp}}(g_0C)>t\}.
\end{align*}

If the Euclidean center and the Euclidean radius of $C$ are $e_C$ and $r_C$ respectively, then the Euclidean center and the Euclidean radius of $g_0C$ are $re_C+p$ and $rr_C$ respectively. We deduce from (\ref{equivalent area}) that
\begin{align*}
N_t(\Gamma(C_1)\cap\mathcal{T})=\left\{C\in \Gamma_0 (C_0): C\cap g_0^{-1}\mathcal{T}\neq \emptyset, r_C>(\operatorname{Im}e_C+\frac{p_2}{r})\beta(t)\right\},
\end{align*}
with $\beta(t)=\frac{\sqrt{t(4\pi+t)}}{2\pi+t}$.

Recall that we introduced the measure $(\operatorname{Im}z)^{-\alpha}d\omega_{\Gamma}$ on $\mathcal{T}$ to estimate $N_t(\Gamma(C_0)\cap\mathcal{T})$. Similarly, we consider the measure $\left(\operatorname{Im}z+\frac{p_2}{r}\right)^{-\alpha}d\omega_{\Gamma_0}$ on $g_0^{-1}(\mathcal{T})$. We claim that for any Borel set $E\subset \mathcal{T}$, 
\begin{equation}
\int_{z\in g_0^{-1}E}\left(\operatorname{Im}(z)+\frac{p_2}{r}\right)^{-\alpha}d\omega_{\Gamma_0}(z)=\int_{z\in E} (\operatorname{Im}z)^{-\alpha}d\omega_{\Gamma}(z).
\end{equation}

Note that $\{g_0^{*}\nu_{\Gamma, g_0x}:x\in \mathbb{H}^3\}$ is a family of $\Gamma_0$-invariant conformal density of dimension $\delta_{\Gamma_0}=\delta_{\Gamma}=\alpha$. The claim can be verified using this observation together with Lemma \ref{Busemann function}.
 

Repeating the process of getting effective estimate of $N_t(\Gamma(C_0)\cap\mathcal{T})$, we can show that there exists $\rho>0$ such that as $t\to 0$,
\begin{align*}
 N_t(\Gamma(C_1)\cap\mathcal{T})
&=\frac{|\mu^{\operatorname{PS}}_{\Gamma_0,H}|}{\alpha|m^{\operatorname{BMS}}_{\Gamma_0}|}\,\left(\frac{t}{\pi}\right)^{-\frac{\alpha}{2}}\int_{z\in\mathcal{T}}(\operatorname{Im}z)^{-\alpha}d\omega_{\Gamma}(z)+O(t^{-\frac{\alpha}{2}+\rho}) .
\end{align*}
Note that $|m^{\operatorname{BMS}}_{\Gamma_0}| =|m^{\operatorname{BMS}}_{\Gamma}|$
by Lemma \ref{total mass of BMS}. Moreover, since $\Gamma$ just has rank 1 cusps, we know that   $(\operatorname{Im}z)^{-\alpha}d\omega_{\Gamma}$ agrees with $(\operatorname{Im}z)^{-\alpha}d\mathcal{H}^{\alpha}$ on its support $\overline{\mathcal{P}(\mathcal{T})}$ (\cite{Sullivan}). In conclusion, there exists $\rho>0$ such that as $t\to 0$,
\begin{equation*}
N_t(\mathcal{P}(\mathcal{T}))=\frac{\operatorname{sk}_{\Gamma}(\mathcal{P}_0)}{\alpha|m^{\operatorname{BMS}}_{\Gamma}|}\,\left(\frac{t}{\pi}\right)^{-\frac{\alpha}{2}}\int_{\overline{\mathcal{P}(\mathcal{T})}}(\operatorname{Im}z)^{-\alpha}d\mathcal{H}^{\alpha}(z)+O(t^{-\frac{\alpha}{2}+\rho}).
\end{equation*}
\end{proof}

\end{document}